\documentclass{amsart}

\usepackage{amssymb}
\usepackage{amsmath}
\usepackage{amsthm}
\usepackage{comment}

\def\R{\mathbb{R}}
\def\C{\mathbb{C}}
\def\N{\mathbb{N}}
\def\Z{\mathbb{Z}}

\def\S{\mathbb{S}}

\def\H{\mathbb{H}}
\def\B{\mathbb{B}}

\def\diam{\text{{\rm diam}\,}}
\def\card{\#}
\def\Lbd{\Lambda}
\def\udB{\overline{\dim}_B}
\def\vol{\mathrm{vol}}
\def\h{h_{\mathrm{c}}}
\def\Pa{\mathcal{P}}
\def\Gr{\mathcal{G}}

\newtheorem{theorem}{Theorem}[section]
\newtheorem{theorem*}{Theorem}
\newtheorem{lemma}[theorem]{Lemma}
\newtheorem{proposition}[theorem]{Proposition}
\newtheorem{corollary}[theorem]{Corollary}

\newtheorem{conjecture}[theorem]{Conjecture}
\newtheorem{conjecture*}{Conjecture}
\newtheorem{problem}[theorem]{Problem}

\theoremstyle{definition}
\newtheorem{definition}[theorem]{Definition}
\newtheorem{example}[theorem]{Example}

\theoremstyle{remark}
\newtheorem*{remark*}{Remark}

\begin{document}

\title
[Critical exponent, dimension of limit set
and convex core entropy]
{The critical exponent, the Hausdorff dimension of the limit set 
and the convex core entropy of a Kleinian group}

\author{Kurt Falk}
\address{Universit\"at Bremen, FB 3 - Mathematik, 
Bibliothekstra{\ss}e 1, 28359 Bremen, Germany}
\email{khf@math.uni-bremen.de}
\author{Katsuhiko Matsuzaki}
\address{Department of Mathematics, School of Education,
Waseda University, Nishi-Waseda 1-6-1, Shinjuku, Tokyo 169-8050, Japan}
\email{matsuzak@waseda.jp}

\date{}

\begin{abstract}
In this paper we study the relationship between three numerical 
invariants associated to a Kleinian group, namely the critical exponent, 
the Hausdorff dimension of the limit set and the convex core entropy,
which coincides with the upper box-counting dimension of the limit set.
The Hausdorff dimension of the limit set is naturally bounded below
by the critical exponent and above by the convex core entropy. We 
investigate when these inequalities become strict and when they 
are equalities.
\end{abstract}

\thanks{The authors were supported by JSPS Grant-in-Aid for Scientific 
Research (B) \#20340030}


\maketitle

\section{Introduction and statement of results}

In this paper we study the relationship between three numerical 
invariants associated to a Kleinian group, namely the critical exponent, 
the Hausdorff dimension of the limit set and the convex core entropy,
which is shown here to coincide with the upper box-counting dimension,
also referred to as upper Minkowski dimension, of the limit set.
While the relationship between the first two has already been
addressed in various ways, the convex core entropy of a Kleinian group 
is the new element in the picture.
The convex core entropy $\h(\Lbd(G))$ of a Kleinian group $G$ 
acting on the unit ball $\B^{n+1}$
is computed in a similar way to the volume entropy of the hyperbolic manifold
given by $G$, with the essential difference that one restricts the universal
cover of the manifold, in this case simply hyperbolic space, 
to the $\varepsilon$-neighbourhood of the convex hull $H(\Lbd(G))$ of the 
limit set $\Lbd(G)$ of $G$.
The name convex core entropy is owed to the fact that the lift of the 
convex core of a hyperbolic manifold to the universal cover is precisely
the convex hull of the limit set of the underlying Kleinian group.
The convex core entropy can also be written as the critical exponent $\Delta(X)$
of an `extended Poincar\'e series', given in terms of a uniformly distributed 
set $X$ within $H(\Lbd(G))$.
In this section we shall denote the above invariants simply by 
$\delta$, $\dim_H$, $\h$ and $\udB$, respectively,
when it is clear from the context what Kleinian group they are 
associated with. Also, $\dim_B$ will refer to the box-counting
dimension when it exists.
The most general relationship between them, formulated as 
Theorem~\ref{basicinequality} in the main text, is the following.

\begin{theorem*}
\label{basineqintro}
For any non-elementary Kleinian group we have
$$\delta \leq \dim_H \leq \h = \overline{\dim}_B.$$
\end{theorem*}

Here, only the last equality is new, while the middle inequality 
$\dim_H \leq \overline{\dim}_B$ is true in general
and the first inequality is 
well-known as a consequence of Bishop and Jones \cite{bj97}. 
Thus, the Hausdorff dimension of the limit set is naturally bounded 
below and above by the critical exponent and the convex core entropy, 
respectively. The main goal of this paper is to investigate for what classes of 
Kleinian groups the above inequalities become strict and when they are 
equalities. For instance, in the case of non-elementary geometrically finite 
Kleinian groups, it is well-known that $\delta = \dim_H$, and also that 
$\dim_H = \dim_B$ (see Stratmann and Urba{\'n}ski \cite{stur96}), 
thus making all four invariants coincide (Proposition \ref{geomfiniteequal}).
Moreover, when a Kleinian group is acting on the $3$-dimensional hyperbolic 
space $\B^3$, Bishop \cite{bis96}, \cite{bis97} has shown $\delta=\dim_B$ 
if it is non-elementary and analytically finite and if
its limit set has null $2$-dimensional measure (see also \cite{bj97}).
The arguments in \cite{bis97} use the so-called Whitney exponent of a closed
set in the plane (see also \cite{bt54} and \cite{tri81}), which is related in 
spirit to our convex core entropy but is defined in the boundary of hyperbolic 
space, as opposed to the interior (see also our remark after 
Theorem~\ref{basicinequality}).
Taking the recent solution of the Ahlfors measure conjecture (which follows 
from the Tameness Conjecture proven in \cite{ag04} and \cite{cg06}) into 
consideration, the above results imply that all finitely generated 
non-elementary Kleinian groups acting on $\B^3$ satisfy $\delta=\dim_B$.
We shall state partial results on characterising the occurrence of
strict inequalities or equalities between the invariants from
Theorem~\ref{basineqintro}, and will also produce interesting classes of 
examples to further motivate conjectures emerging from such partial results.

The following result, formulated as Theorem~\ref{lwr-bd-cx-hull-entr}, is
the first natural observation on $\h$, and is reminiscent of the 
well-known fact that $\delta \geq n/2$ when the group contains a
parabolic element of maximal rank $n$.

\begin{theorem*}
For any Kleinian group with the property that the convex
hull of its limit set contains some horoball, we have $\h \geq n/2$.
\end{theorem*}

Another question that arises in connection with $\h$ is under what
circumstances we have $\h < n$. Recall that the classical volume
entropy of any $(n+1)$-dimensional hyperbolic manifold is equal to $n$,
and thus $\h = n$ for any Kleinian group of the first kind (i.e. with
the limit set being everything).
This can also be seen by the equality in Theorem~\ref{basineqintro}.
We consider non-elementary Kleinian groups with the property that any 
point within the convex hull of their limit set is a bounded distance 
$\ell \geq 0$ away from the boundary of the convex hull; we call this property
\emph{$\ell$-tight}. 
Notice that the tightness condition also implies 
that the group is of the second kind (i.e. not of the first kind).
It is satisfied for instance by any convex 
cocompact group of the second kind and any non-trivial normal subgroup thereof. 
Actually the tightness property can be relaxed in such a way that the resulting 
property is even true for any geometrically finite group of the second kind.
We call such Kleinian groups geometrically tight.
It has been shown in \cite{mat00} that for geometrically tight groups the strict 
inequality $\dim_H < n$ holds. Here, we show in Theorems~\ref{bdneighbouring} 
and \ref{porosityestimate} that the same holds for $\h=\udB$.
Note that tightness implies porosity of the limit set, and hence it is natural 
to have some relation to $\udB$.

\begin{theorem*}
There is a constant $s(n,\ell)<n$ depending only on $n$ and $\ell$ such that
any Kleinian group with $\ell$-tight convex core satisfies $\h < s(n,\ell)$. 
\end{theorem*}

Next we investigate the question when $\delta < \dim_H$.
This question is particularly interesting, since it asks for what
hyperbolic manifolds the recurrent dynamics within the convex core
is smaller dimension-wise than the total dynamics.
The interpretation of $\delta$ as the dimension of recurrent dynamics
is due to Bishop and Jones~\cite{bj97} who have shown that for all
non-elementary Kleinian groups, $\delta$ coincides with the Hausdorff
dimension of the conical limit set. 
Brooks \cite{bro85} has shown that for any convex cocompact Kleinian group 
satisfying $\delta > n/2$ and any of its non-trivial normal subgroups, the 
bottoms of the $L^2$-spectra coincide if and only if the quotient is amenable.
Using this and the well-known correspondence between $\delta$ and
the bottom of the $L^2$-spectrum, we can formulate the following 
result, stated as Theorem~\ref{normalsubgroup} in the text.
In fact, a new proof of Brooks' result by Stadlbauer~\cite{sta11}
shows that the assumption $\delta > n/2$ can be dropped.

\begin{theorem*}
\label{normalsubgroup-intro}
Let $G$ be a convex cocompact Kleinian group and $N$ some non-trivial 
normal subgroup. 
Then, $\delta(N)<\h(\Lbd(N))(=\dim_H(\Lbd(N)))$
precisely when $G/N$ is non-amenable.
\end{theorem*}

In the context of our work, this constitutes a first instance where 
$\delta < \dim_H = \h$. However, examples of Kleinian groups with
$\delta < \dim_H$ were first given by Patterson \cite{pat79}, \cite{pat83} 
(see also \cite{fs04} for further discussion and references). 
Theorem~\ref{normalsubgroup-intro} can be applied to a class of examples, 
also discussed in \cite{fs04}, where $G$ is a Schottky group
and $N$ is a normal subgroup of $G$ so that
$G/N$ is isomorphic to a Schottky subgroup $G_1$ of $G$ 
and to the isometry group of the manifold given by $N$. 
Therefore, assuming $G_1$ to be non-elementary and thus 
non-amenable ensures the existence of a dimension gap 
of the type $\delta < \dim_H = \h$ for $N$. 
Since the non-amenability of $G_1$, that is,
of its Cayley graph (which can be embedded in the convex core
of the manifold associated to $N$ by taking suitable edge lengths)
is essential in producing the dimension gap,
it is natural to attempt a generalisation of 
Theorem~\ref{normalsubgroup-intro} along
the following lines (Conjecture~\ref{non-amenab-conj}).
The main idea is that $N$ is not assumed to be a normal subgroup
of some group $G$, so that the non-amenability condition is imposed on a
discrete structure given by the pants decomposition of the 
hyperbolic surface given by $N$. 
This also means that, for this problem, we restrict our attention
to the two-dimensional case. Formulating a similar conjecture in
higher dimensions depends on finding a notion of manifold decomposition which
is in some way analogous to the pants decomposition of a hyperbolic surface.

\begin{conjecture*}
\label{non-amenab-conj-intro}
Let $S$ be a hyperbolic surface given by some Fuchsian group.
Assume $S$ is of infinite type with strongly bounded geometry, 
and let $\Gr$ be the graph associated to some 
uniform pants decomposition of $S$. 
Then, $\delta < \h$ if and only if $\Gr$ is non-amenable.
\end{conjecture*}

As usual, a pair of pants is a complete hyperbolic surface with
geodesic boundary whose interior is homeomorphic to the complement
of three points in the 2-sphere. A hyperbolic surface has \emph{strongly 
bounded geometry} if it admits a pants decomposition such that the 
lengths of the decomposing closed geodesics are uniformly bounded away 
from $0$ and $\infty$. Such a pants decomposition is called \emph{uniform}. 
The graph $\Gr$ associated to a pants decomposition is defined naturally 
by letting the pants be vertices and connecting two vertices by an
edge whenever the involved pants are adjacent. 
Due to the method of proof that we use, involving Cheeger constants of
surfaces and isoperimetric constants of graphs, we can only partially
address Conjecture~\ref{non-amenab-conj-intro}, obtaining the following
results stated as Theorem~\ref{amen-implies-equal} and 
Theorem~\ref{non-amen-equiv-delta<1} further down.
The assumptions are the same as in Conjecture~\ref{non-amenab-conj-intro}.

\begin{theorem*}
\label{amen-implies-equal-intro}
If $\Gr$ is uniformly strongly amenable, then all $\delta$, $\dim_H$ and $h_c$
are either equal to $1$ or less than $1$ at the same time.
\end{theorem*}

\begin{theorem*}
\label{non-amen-equiv-delta<1-intro}
If $\Gr$ is non-amenable, then $\delta < 1$.
\end{theorem*}

Uniform strong amenability is a stronger notion of amenability for graphs which 
coincides with the usual amenability for Cayley graphs of finitely generated
groups (see Section~\ref{amenab-pants-section} for more details).
As a consequence of these two theorems, 
we obtain a partial answer to the problem posed in 
Conjecture~\ref{non-amenab-conj-intro}.
Under the assumptions of Conjecture~\ref{non-amenab-conj-intro},
if $\Gr$ is homogeneous and $\h=1$, then $\delta < \h = 1$ precisely when
$\Gr$ is non-amenable.
Here, a graph is called homogeneous if its automorphism group acts 
transitively. For homogeneous graphs all notions of amenability considered in
Section~\ref{amenab-pants-section} coincide, and again, 
Cayley graphs of finitely generated groups are homogeneous.

We provide further evidence for the validity of 
Conjecture~\ref{non-amenab-conj-intro} in the form of special classes 
of examples for which the statement of the conjecture holds.
In Section~\ref{examples-dim-gap} we provide examples of 
hyperbolic surfaces for which $\delta<\h=1$ and even 
$\delta < \h < 1$. The fundamental piece of our construction is a
pair of pants whose geodesic boundaries have sufficiently large length $\ell$, 
and suitably pasting copies of this one piece together makes tree-like 
hyperbolic surfaces. Then, as $\ell \to \infty$, the bottom of $L^2$-spectra 
of the surfaces tends to the maximum possible value, which implies that
$\delta \to 1/2$, whereas the number of pairs of pants within distance $R$ 
of some base point remains bounded from below
by $a^R$ for some universal constant $a$ larger than $1/2$.
The last condition implies that $\h \geq a$, and thus we get the 
difference between $\delta$ and $\h$.

In order to attack Conjecture~\ref{non-amenab-conj-intro} in its full generality
we expect that further developments of Kesten's work \cite{kes59} on random
walks on graphs (see also \cite{woe00}), as employed by 
Stadlbauer~\cite{sta11} to produce an alternative proof of
Brooks' theorem, will be more instrumental than attempting to modify 
Brooks' initial approach. Progress in this direction has also been made in 
\cite{jae11}, where group-extended Markov systems are considered.

The last part of this paper is concerned with the question when $\dim_H=\h$.
In light of the last equality of Theorem~\ref{basineqintro}, this of 
course translates to the question when $\dim_H = \udB$, and thus 
$\dim_H = \dim_B$, since the first equality trivially implies the existence 
of the box-counting dimension. 
This question has already been addressed both in the general fractal geometric 
context and specifically for Kleinian groups.
It is well-known that the two dimensions coincide for self-similar sets
(see e.g. \cite{fal90} or \cite{mati95} for details). Often, equality holds
when some degree of homogeneity of the considered set is ensured, or
when it supports a sufficiently regular measure, in the sense that
the measure of small balls is comparable to the radius raised to a certain
power, which turns out to be the common value of the two dimensions
(see \cite{sal91}, \cite{you82} and \cite[Theorem~5.7]{mati95}).
For a sufficient condition for equality without control on the actual value
of the dimension, see e.g. \cite{fal89}. Another large class of examples is
that of self-affine sets, for which the dimensions can differ \cite{mcm84},
\cite{lg92}, \cite{gl94}, \cite{kp96a}, \cite{kp96b}, but coincide in a certain 
`generic' sense \cite{fal88}, \cite{bu90}. The situation for limit sets of 
Kleinian groups we have already briefly discussed after 
Theorem~\ref{basineqintro}.

Here, however, we prefer the novel approach via the convex core entropy
$\h$ as it allows formulating a sufficient condition for the above equality 
within hyperbolic space, as opposed to the limit set itself.
We introduce the \emph{bounded type} condition for a uniformly distributed 
set $X$, stating that all cardinalities of intersections of $X$ with balls 
centred at points in $X$ are comparable, independently of centre and radius. 
As a weaker version of this condition, we also introduce the 
\emph{weakly bounded type} condition
by ignoring points of $X$ contained in a family of disjoint horoballs.
Then we develop a Patterson-Sullivan theory for uniformly distributed sets in 
convex hulls of closed sets in the boundary of hyperbolic space. 
This also includes a shadow lemma stated as Theorem~\ref{shadow} 
in the main text. These considerations enable us to prove the following 
result appearing as Theorem~\ref{hausdorff} later on.

\begin{theorem*}
\label{hausdorff-entropy-intro}
Let $\Lbd$ be a closed subset of $\mathbb S^n$ and assume 
there is a uniformly distributed set $X$ of bounded type
in the convex hull $H(\Lbd)$.
Then the $\Delta(X)$-dimensional Hausdorff measure of $\Lbd$ is positive.
In particular, 
$\dim_H (\Lbd) = \h(\Lbd)=\dim_B (\Lbd)$.
\end{theorem*}

Theorem~\ref{hausdorff-entropy-intro} implies that any Kleinian group of 
bounded type, that is, with the property that the convex hull of its limit set 
admits a uniformly distributed set of bounded type, satisfies 
$\dim_H = \h = \dim_B$.
As a consequence of Theorem~\ref{hausdorff-entropy-intro} 
(see Corollary~\ref{bounded-divergent}) we obtain that the Poincar\'e series
$P^s(X,o)=\sum_{x \in X}e^{-sd(x,o)}$ for a uniformly distributed set $X$ of 
bounded type is also of divergence type, meaning that
$P^s(X,o)$ diverges at the critical exponent $s=\Delta(X)$.
This raises the natural question whether the divergence type of the Poincar\'e 
series for $X$ already implies $\dim_H (\Lbd) = \h(\Lbd) = \dim_B (\Lbd)$.
Note, the divergence type property of a uniformly distributed set $X$ 
in the convex hull of the limit set of some Kleinian group $G$ is not to be 
confused with the divergence type property of $G$ itself.

If the bounded type condition is relaxed to a weak version 
(Definition~\ref{weaklybounded}), then we show in Theorem~\ref{wb-packing}
that the convex core entropy of the given closed set $\Lbd$ coincides
with its packing dimension $\dim_P(\Lbd)$. In particular, any Kleinian
group all of whose parabolic fixed points are bounded, and whose limit set gives
a uniformly distributed set of weakly bounded type in its convex hull, 
satisfies the assumptions of Theorem~\ref{wb-packing}
(see Corollary~\ref{wbimplydivergence}). Geometrically finite groups clearly
fall into this category.

At this point it is perhaps useful to remark that 
Theorem~\ref{hausdorff-entropy-intro} could also be obtained 
using ideas inspired by the way Bishop and Jones \cite{bj97} showed  
that for non-elementary Kleinian groups $\delta$ equals the 
Hausdorff dimension of the conical limit set, perhaps with the 
additional subtlety that the group structure, which is essential
in their argument, must be replaced with our bounded type condition
or similar. 
However, we preferred the approach via the Patterson-Sullivan 
theory of uniformly distributed sets, as it constitutes 
a useful tool for further developments.

Finally, Section~\ref{bd-type-cond} is devoted to the bounded type condition 
which, as seen above, implies $\dim_H=\h=\dim_B$ for Kleinian groups.
The condition holds for groups of the first kind and for convex cocompact
groups, but not for geometrically finite groups with parabolic elements.
The weakly bounded type condition can be used for including these groups.
We also observe that it is inherited to normal subgroups and that it is 
invariant under rough isometries. We propose the problem of asking whether 
it is even invariant under quasi-isometries. 
It is, however, possible to construct examples
of groups of bounded type which are not contained in any of the above 
categories. This construction uses the work of Eskin, Fisher and 
Whyte~\cite{efw} on the Diestel-Leader
graph; there is a locally finite graph whose automorphism group
acts transitively but which is \emph{not} quasi-isometric to any Cayley 
graph of a finitely generated group. In fact, we expect a large class 
of geometrically infinite groups to have the (weakly) bounded type property 
(Conjecture~\ref{bd-from-bdry-bd-type}).

\begin{conjecture*}
\label{bd-from-bdry-bd-type-intro}
A geometrically tight Kleinian group of divergence type is of 
weakly bounded type. 
\end{conjecture*}

Thus, it may well be that Fuchsian groups with the tightness 
property form the first class of groups beyond those we could 
treat here, which will be shown to satisfy 
Conjecture~\ref{non-amenab-conj-intro} in its full generality.

\medskip
\noindent
{\bf Acknowledgements.} The first author is indebted to Pekka Tukia for
the many interesting and eventually fruitful conversations on the idea of 
an extended Poincar\'e series and its critical exponent. 
We thank the referee of an earlier version of this manuscript for pointing out 
the connections to porosity and to the upper box-counting dimension. 
This has significantly increased the information content of our paper.

\section{Preliminaries}

Let $(\B^{n+1},d)$, $n\geq 1$, be the unit ball model of $(n+1)$-dimensional 
hyperbolic space with the hyperbolic distance $d$. Thus, the $n$-dimensional 
unit sphere $\S^n$ is the boundary at infinity of hyperbolic space. 
Let $\vol(\cdot)$ denote hyperbolic volume and $A(\cdot)$ hyperbolic area. 
The hyperbolic metric also allows computing lengths $\ell(\cdot)$ of curves in 
the usual way. \emph{Kleinian groups} are discrete subgroups of the group of 
orientation preserving isometries of hyperbolic space. The quotient
$M=M_G=\B^{n+1}/G$ of $(n+1)$-dimensional hyperbolic space through a torsion free
Kleinian group, that is, a group without elliptic elements, is an 
$(n+1)$-dimensional hyperbolic manifold. For ease of notation, we keep
the notation $d$ for the projected metric on $M$. If $n=1$, then $M$ is a 
hyperbolic surface and $G$ is referred to as a \emph{Fuchsian group}.
If $G$ is infinitely generated, we call the resulting hyperbolic surface
\emph{of infinite type}.

The \emph{limit set} $\Lbd(G)$ of a Kleinian group $G$ is the set of accumulation
points of an arbitrary $G$-orbit, and is a closed subset of $\S^n$. If $\Lbd(G)$
consists of more than two points, then it is uncountable and perfect, and
$G$ is called \emph{non-elementary}. Now let $\Lbd$ be some closed subset 
of $\S^n$. The hyperbolic convex hull of the union of all geodesics both
of whose end points are in $\Lbd$ is called the \emph{convex hull of} $\Lbd$, and
is denoted by $H(\Lbd)$. For a positive constant $\varepsilon >0$, the 
$\varepsilon$-neighbourhood of $H(\Lbd)$ is denoted by $H_\varepsilon(\Lbd)$.
When $\Lbd = \Lbd(G)$ for some Kleinian group $G$, the 
group acts discontinuously on the convex hull of $\Lbd(G)$, and the 
quotient $C(M_G) = H(\Lbd(G))/G$ is called the \emph{convex core} of 
$M_G=\B^{n+1}/G$.
Equivalently, the convex core is the smallest convex subset of $M_G$
containing all closed geodesics of $M_G$. A non-elementary 
Kleinian group $G$ is called \emph{convex cocompact} if the convex core
$C(M_G)$ is compact, and \emph{geometrically
finite} if some $\varepsilon$-neighbourhood 
$C_\varepsilon(M_G)=H_\varepsilon(\Lbd)/G$ of
the convex core has finite hyperbolic volume.

A point $\xi \in \Lbd(G)$ is a \emph{conical limit point} of $G$ if there exists
a ball $B$ in $\B^{n+1}$ such that for every geodesic ray $r$ towards
$\xi$ we have $g(B)\cap r \neq \emptyset$ for infinitely many distinct 
$g \in G$. The set of all conical limit points of $G$ is called the
\emph{conical limit set} and is denoted by $\Lbd_c(G)$.

A \emph{horoball} is an $(n+1)$-dimensional open ball in $\B^{n+1}$ whose
bounding sphere, called \emph{horosphere}, is tangent to $\S^n$. If $p$
is the fixed point of some parabolic element of the Kleinian group $G$,
then any horoball at $p$ is invariant under the stabiliser 
$\mathrm{Stab}_G(p)$ of $p$ in $G$. Moreover, there is a horoball $D_p$
at $p$ so that $D_p \cap g(D_p) = \emptyset$ for all 
$g \in G \setminus \mathrm{Stab}_G(p)$. Clearly, any other horoball contained
in $D_p$ also has this property. The projection of such a horoball
$D_p$ to $M=\B^{n+1}/G$ is called a \emph{cusp neighbourhood}. 
We can always find a cusp neighbourhood of volume one
for each $p$, which is called the canonical cusp neighbourhood.
We say that $p$ is a \emph{bounded parabolic fixed point} of $G$ if
the intersection of $C(M)$ with $\partial D_p / \mathrm{Stab}_G(p)$ is
compact, where $\partial D_p$ denotes the bounding horosphere of $D_p$.
A theorem of Beardon and Maskit \cite{bm74} asserts that $G$ is 
geometrically finite if and only if its limit set is the disjoint
union of $\Lbd_c(G)$ and the (at most countable) set of bounded parabolic
fixed points.

For a Kleinian group $G$ and points $x,z \in \B^{n+1}$, 
the \emph{Poincar{\'e} series} with exponent $s > 0$ is given by
$$
P^s(Gx,z) = \sum_{g \in G} e^{-s \, d(g(x),z)}.
$$
The \emph{critical exponent} $\delta=\delta(G)$ of $G$ is
\begin{eqnarray*}
\delta(G)&=&
\inf \, \{ s>0:P^s(Gx,z)<\infty\} \\
&=&\limsup_{R\to\infty} \, \frac{\log \card (B_R(z) \cap G(x))}{R}\, ,
\end{eqnarray*}
where $B_R(z)$ is the hyperbolic ball of radius $R$ centred at $z$ and
$\card(\cdot)$ denotes the cardinality of a set. By the triangle inequality,
$\delta$ does not depend on the choice of $x,z \in \B^{n+1}$. 
If $G$ is non-elementary, then $0<\delta\leq n$. 
Also, Roblin~\cite{rob02} showed that the above upper limit is in fact a limit.
$G$ is called \emph{of convergence type} if $P^\delta(Gx,z)<\infty$,
and \emph{of divergence type} otherwise.

A family of finite Borel measures $\{\mu_z\}_{z \in \B^{n+1}}$ on $\S^n$ 
is called \emph{$s$-conformal measure} for $s>0$ if $\{\mu_z\}$
are absolutely continuous to each other and, for each $z \in \B^{n+1}$
and $\xi \in \S^n$,
$$
\frac{d\mu_z}{d\mu_o} (\xi) = k(z,\xi)^s,
$$
where $k(z,\xi) = (1-|z|^2)/|\xi - z|^2$ is the \emph{Poisson kernel}
and $o$ the origin in $\B^{n+1}$.
Recall Patterson's construction \cite{pat76}, \cite{pat87} of 
$\delta$-conformal measures $\{\mu_z\}$ 
supported on the limit set $\Lbd(G)$ of a Kleinian group $G$,
which satisfy $g^*\mu_{g(z)}=\mu_z$ for every $z \in \B^{n+1}$ and for every 
$g \in G$. This $G$-invariance property can be rephrased as follows:
for every $g \in G$ and any measurable 
$A \subset \S^n$ we have
$$
\mu_o(g(A)) = \int_A |g'(\xi)|^\delta d\mu_o(\xi).
$$
Here, $|g'(\xi)|$ is the unique positive number called the linear stretching 
factor so that $g'(\xi)/|g'(\xi)|$ is orthogonal.

For further fundamentals on Kleinian groups and hyperbolic manifolds
we refer the reader for instance to the books \cite{bea83}, \cite{mas89}, 
\cite{rat94}, \cite{mt98}, \cite{kap01} and \cite{nic89}.

We assume familiarity with the notion of Hausdorff dimension
and packing dimension of a set in $\R^n$. 
The latter is defined in a similar manner to the Hausdorff dimension,
just using packings of disjoint balls centred in the given set instead of
coverings of the set (for further details see \cite{fal90} or \cite{mati95}).
We only recall that the \emph{upper box-counting dimension}
(also referred to as \emph{upper Minkowski dimension}) of $\Lbd \subset \R^n$ 
is given by
$$
\udB(\Lbd) = \limsup_{\varepsilon \to 0} 
\frac{\log N_\varepsilon(\Lbd)}{-\log \varepsilon},
$$
where $N_\varepsilon(\Lbd)$ is the smallest number of sets of diameter
at most $\varepsilon$ that cover $\Lbd$, or the largest number of
disjoint balls of radius $\varepsilon$ with centres in $\Lbd$. If 
the \emph{lower box-counting dimension} $\underline{\dim}_B(\Lbd)$ of $\Lbd$, obtained by replacing the 
limit superior with a limit inferior in the above definition, coincides with
$\udB(\Lbd)$, then the \emph{box-counting dimension} (also called
\emph{Minkowski dimension}) of $\Lbd$ exists
and is denoted $\dim_B(\Lbd)$ 
(see e.g. \cite{fal90} for more details).

The upper box-counting dimension is related to the porosity of a set.
We say that a subset $\Lbd \subset \R^n$ is {\it $c$-porous} if there is a 
constant $c>0$ so that for each $x \in \Lbd$ and all $r > 0$ the ball $D_r(x)$ 
of centre $x$ and radius $r$ contains a point $y$ with the Euclidean 
distance $d\,_{\R^n}(y,\Lbd) \geq cr$.
If $\Lbd$ is $c$-porous, then
$$
\overline{\dim}_B (\Lbd) \leq n-C_nc^n
$$
for some constant $C_n>0$ depending only on $n$
(see e.g. \cite{kr97}).

As usual, the limsup set of a countable collection $(A_i)_{i \in \N}$ of sets 
consists of all points which lie in infinitely many member sets of the given 
collection:
$$
\limsup \{ A_i : i\in\N \} = 
\{ x : x\in A_i \mbox{ for infinitely many } i\in\N \}.
$$

A homeomorphism $f:X \to X$ of a metric space $(X,d)$ into itself is called a 
{\em $K$-quasi-isometry} if there exist constants $K \geq 1$ and $L \geq 0$ 
such that
$$
\frac{1}{K}d(x,y)-\frac{L}{K} \leq d(f(x),f(y)) \leq K d(x,y)+L
$$
for any $x,y \in X$. A $K$-quasi-isometry is called a 
{\em rough isometry} if $K=1$.

\section{Critical exponent, Hausdorff dimension and convex core entropy}
\label{3}

For a compact Riemannian manifold $(X,g)$ the volume entropy is defined as
$$ 
\lim_{R \to \infty} 
\frac{\log \vol_g\,B_R(z)}{R},
$$
where $B_R(z)$ is the ball of radius $R$ centred at some $z \in \widetilde{X}$,
with $\widetilde{X}$ the universal cover of $X$, and $\vol_g$ is the 
volume element induced by the Riemannian metric $g$.
It is well-known that in this case the limit exists and is independent 
of the choice of $z \in \widetilde{X}$. Also, it is clear that for 
any (not necessarily compact) hyperbolic manifold
$\B^{n+1}/G$, with $G$ some Kleinian group, the volume entropy just
equals $n$. In the context of our work it is, however, more useful
to consider the following notion of volume entropy which takes into account 
the convex hull of the limit set $\Lbd(G)$ of $G$.

\begin{definition}
For a closed set $\Lambda$ in $\S^n$,
we define the \emph{convex core entropy} as
$$
\h(\Lbd) = \limsup_{R \to \infty} 
\frac{\log \vol\,(B_R(z) \cap H_\varepsilon(\Lbd))}{R} ,
$$
where  
$B_R(z)$ is the ball of centre $z$ and radius $R$ 
in $\B^{n+1}$. For a hyperbolic manifold $M_G=\B^{n+1}/G$, we define 
the convex core entropy $\h(M_G)$ to be $\h(\Lbd(G))$.
\end{definition}

The definition is independent of the choice of $z \in \B^{n+1}$ 
and of a positive constant $\varepsilon>0$. The latter is seen from
Proposition \ref{cc-entr-is-delta-X} below.
The simplest reason why we need to take the $\varepsilon$-neighbourhood is that,
in the degenerate case such as a Fuchsian group viewed as a Kleinian group,
the volume of the convex hull is zero and there is no use considering 
its entropy. In the other case,
by replacing $H_\varepsilon(\Lbd)$ with $H(\Lbd)$, we can define the `exact' 
convex core entropy, which is of course not greater than $\h(\Lbd)$, but 
a priori we do not know whether they are equal or not.
If $H(\Lbd)$ does not have any thin part such as a cusp or an arbitrarily 
narrow strait, then the two values are the same.
From the definition, we also have that 
$$
0 \leq \h(\Lbd) \leq \h(\S^n) = n
$$ 
for every closed set $\Lbd \subset \S^n$.

\begin{definition}
We call a discrete set $X=\{x_i\}_{i=1}^{\infty}$ in the convex hull 
$H(\Lbd) \subset \B^{n+1}$ \emph{uniformly distributed} if the following two 
conditions are satisfied:
\begin{itemize}
\item [(i)]
There exists a constant $M<\infty$ such that, for every point 
$z \in H(\Lbd)$, there is some $x_i \in X$ such that $d(x_i,z)\leq M$;
\item [(ii)]
There exists a constant $m>0$ such that, any distinct points 
$x_i$ and $x_j$ in $X$ satisfy $d(x_i,x_j) \geq m$.
\end{itemize}
\end{definition}

\begin{remark*}
In any convex hull $H(\Lbd) \subset \B^{n+1}$, we can always choose some 
uniformly distributed set. For example, consider the orbits of several points 
under some cocompact Kleinian group acting on $\B^{n+1}$ 
(or on $\B^{k} \subset \B^{n+1}$ in the degenerate case)
and take the intersection with $H(\Lbd)$. This gives uniformly distributed 
points in a thick part of $H(\Lbd)$. If we choose sufficiently many orbit 
points, they are distributed in a large portion of $H(\Lbd)$ and
the rest of $H(\Lbd)$ is a thin part of simple structure. 
Then we can fill in points suitably `by hand' in the thin part
to obtain a uniformly distributed set in $H(\Lbd)$.
In the upper half-space model $\H^{n+1}$, we may also use the set of dyadic 
points $\{c_k(m_1,\ldots,m_n)\}$ as in the proof of 
Theorem~\ref{basicinequality} below instead of the orbit points.
We propose the problem of existence of uniformly distributed sets in the
more general setting of metric spaces.
\end{remark*}

For a Kleinian group $G$ acting on $\B^{n+1}$, the critical exponent 
$\delta(G)$ is defined to be the infimum of all exponents $s>0$ such that the 
Poincar\'e series $\sum_{g \in G}e^{-s\,d(g(x),z)}$ with $x,z \in \B^{n+1}$ 
converges. Similarly, we define the critical exponent for a uniformly 
distributed set as follows.

\begin{definition}
For a uniformly distributed set $X$, we define the 
\emph{Poincar\'e series} with exponent $s>0$ and reference point 
$z \in \B^{n+1}$ by
$$
P^s(X,z)=\sum_{x \in X}e^{-s\,d(x,z)}.
$$
The \emph{critical exponent} for $X$ is
$$
\Delta = \Delta(X)=\inf \{s>0 \mid P^s(X,z)<\infty\}.
$$
The Poincar\'e series for $X$ (or more simply
$X$ itself) is of \emph{convergence type} if $P^\Delta(X,z)<\infty$, and of
\emph{divergence type} otherwise.
\end{definition}

As in the case of groups (see e.g. \cite{nic89}), it is not too difficult to
see that
$$
\Delta(X) = \limsup_{R \to \infty} \frac{\log \card (B_R(z) \cap X)}{R}
$$
for any $z \in \B^{n+1}$, where $\card (\cdot)$ denotes the cardinality of a set.

\begin{remark*}
For the orbit $X=Gz$ of a Kleinian group $G$, 
Roblin \cite{rob02} proved that the above 
limit superior is actually the limit. 
One can ask what condition ensures that this is also the case 
for our $\Delta(X)$ as well as for the convex core entropy $h_c(\Lambda)$.
Later, in Theorem~\ref{basicinequality}, we will see that $h_c(\Lambda)$ 
coincides with the upper box-counting dimension $\overline{\dim}_B(\Lbd)$. 
Thus the problem is transferred to asking when $\overline{\dim}_B(\Lbd)$ 
is equal to the lower box-counting dimension $\underline{\dim}_B(\Lbd)$.
\end{remark*}

\begin{proposition}
\label{cc-entr-is-delta-X}
Given a closed set $\Lbd \subset \S^n$, 
we have that
$$
\h(\Lbd) = \Delta(X)
$$
for any uniformly distributed set $X$ in the convex hull $H(\Lbd)$.
\end{proposition}

\begin{proof}
By the definition of uniformly distributed set, we see that
the union of all hyperbolic (closed) balls $B_{M+\varepsilon}(x_i)$ 
of radius $M+\varepsilon$ centred at $x_i \in X$ covers the 
$\varepsilon$-neighbourhood $H_\varepsilon(\Lambda)$ of the convex hull. 
Hence we have
$$ 
\vol\,(B_R(z) \cap H_\varepsilon(\Lambda)) \leq 
\vol\,(B_{M+\varepsilon}) \cdot \#(B_R(z) \cap X),
$$
where $B_{M+\varepsilon}$ denotes any hyperbolic ball of radius 
$M+\varepsilon$. This gives the inequality $h_c(\Lambda) \leq \Delta(X)$.

On the other hand, to show the converse inequality, we set 
$m'=\min\{m/2,\varepsilon\}$ and take the hyperbolic (open) ball 
$B_{m'}(x_i)$ centred at $x_i \in X$ with radius $m'$.
Then $B_{m'}(x_i)$ is contained in $H_\varepsilon(\Lambda)$ and does not 
intersect any ball of the same radius centred at some other point in $X$.
From this observation, we see that
$$ 
\vol\, (B_{R+m'}(z) \cap H_\varepsilon(\Lambda)) \geq \vol\, (B_{m'}) \cdot
\#(B_R(z) \cap X),
$$
where $B_{m'}$ denotes any hyperbolic ball of radius $m'$.
This yields $h_c(\Lambda) \geq \Delta(X)$, and hence $h_c(\Lambda)=\Delta(X)$.
\end{proof}

For a convex cocompact Kleinian group $G$,
we can take the orbit $G(z)$ of some $z \in H(\Lambda(G))$ as a uniformly
distributed set $X$ in the convex hull of the limit set $\Lambda(G)$.
In this case, $\Delta(X)=\delta(G)$, and hence $\h(\Lbd(G)) = \delta(G)$
by the above proposition. We can extend this statement to geometrically finite 
Kleinian groups as is seen later.

Basic inequalities for our numerical invariants are the following.

\begin{theorem}\label{basicinequality}
For any closed set $\Lbd \subset \S^n$ we have 
$$
\dim_H(\Lbd) \leq \overline{\dim}_B(\Lbd)=\h(\Lbd).
$$
When $\Lbd=\Lbd(G)$ is the limit set of some non-elementary Kleinian group $G$, 
we have that
$$
\delta(G) \leq \dim_H (\Lbd(G)) \leq \overline{\dim}_B(\Lbd(G))=\h(\Lbd(G)).
$$
\end{theorem}

\begin{proof}
We prove that the convex core entropy $\h(\Lbd)$ of $\Lbd$ coincides with its 
upper box-counting dimension $\udB(\Lbd)$, thus rendering the first inequality 
of the theorem trivial (see e.g. \cite{fal90}).
Let $\Lambda\ (\neq \S^n)$ be a closed subset of $\S^n$. 
By using the upper half-space model $\H^{n+1}$ for hyperbolic space,
we may assume that $\Lambda$ is a compact subset of $\R^n$.
A uniformly distributed set $X$ in the convex hull $H(\Lambda)$
can be given as follows.
For each $(m_1,\ldots,m_n) \in \Z^n$ and each $k \in \Z$, we
define the following dyadic square of size $2^{-k}$ in $\R^n$
and a cuboid over it in $\H^{n+1}$ together with its `upper centre':
\begin{eqnarray*}
I_k(m_1,\ldots, m_n)&=&\left[\frac{m_1}{2^k},\frac{m_1+1}{2^k}\right] \times 
\cdots \times \left[\frac{m_n}{2^k},\frac{m_n+1}{2^k}\right]
\subset \R^n;\\
\tilde I_k(m_1,\ldots, m_n)&=&\left[\frac{m_1}{2^k},\frac{m_1+1}{2^k}\right] 
\times \cdots \times \left[\frac{m_n}{2^k},\frac{m_n+1}{2^k}\right]
\times \left[\frac{1}{2^{k+1}},\frac{1}{2^{k}} \right]
\subset \H^{n+1};\\
c_k(m_1,\ldots, m_n)&=&\left(\frac{m_1+1/2}{2^k},\ldots, \frac{m_n+1/2}{2^k}, 
\frac{1}{2^{k}}\right) \in \H^{n+1}.
\end{eqnarray*}
If $H(\Lambda) \cap  \tilde I_k(m_1,\ldots, m_n) \neq \emptyset$, then we 
choose the nearest point 
$$
x_k(m_1,\ldots, m_n) \in H(\Lambda) \cap  \tilde I_k(m_1,\ldots, m_n)
$$ 
to the upper-centre $c_k(m_1,\ldots, m_n)$. 
Set 
$$
X=\{x_k(m_1,\ldots, m_n) \mid k \in \Z,\ (m_1,\ldots, m_n) \in \Z^n\}.
$$
We see that $X$ is a uniformly distributed set in $H(\Lambda)$, except for the 
eventuality that there is a cluster of
points $x_k(m_1,\ldots, m_n)$ in a thin part of $H(\Lbd)$ 
that are very close to each other. 
If this is the case, we only have to select one such point from each 
cluster as a member of $X$ and remove the rest in order 
to make $X$ uniformly distributed.

For a fixed integer $k$, let $N_k(\Lambda)$ be the number of dyadic squares 
$\{I_k(m_1,\ldots, m_n)\}$ that intersect $\Lambda$. Then the upper box-counting 
dimension of $\Lambda$ is given by
$$
\udB (\Lambda)= \limsup_{k \to \infty} 
\frac{\log N_k(\Lambda)}{k \log 2}
$$
(see e.g. \cite[Chapter 3]{fal90}).

The convex core entropy is defined by
$$
\h(\Lbd) = \limsup_{R \to \infty} 
\frac{\log \vol\,(B_R(z) \cap H_\varepsilon(\Lbd))}{R} ,
$$
where  
$B_R(z)$ is the ball of centre $z$ and radius $R$ 
in $\H^{n+1}$. Here we may replace the point $z$ with 
any compact subset $Z$ of $\H^{n+1}$ to obtain the same quantity.
We consider $Z=W \times \{\tau\} \subset \H^{n+1}$, 
where $W \subset \R^n$ is a sufficiently large closed ball containing $\Lambda$,
and $\tau>0$ is sufficiently large, that is, above the highest 
$(n+1)$-th coordinate of $H_\varepsilon(\Lbd)$.
By using the fact that the distance between $Z$ 
and $W \times \{2^{-k}\}$ is comparable to $k \log 2$, we obtain
\begin{eqnarray*}
\h(\Lbd) &=& \limsup_{k \to \infty} \frac{\log \vol\,(\R^n \times [2^{-k},\infty) 
\cap H_\varepsilon(\Lbd))}{k \log 2}\\ 
&=&\limsup_{k \to \infty} \frac{\log \vol\,(\R^n \times [2^{-(k+1)},2^{-k}] 
\cap H_\varepsilon(\Lbd))}{k \log 2}.
\end{eqnarray*}
The latter equality always holds in the case of exponential growth.
Also, in general, the limit superior with respect to the continuous variable 
$R \to \infty$ is greater than or equal to that with respect to the discrete 
sequence $k \to \infty$, but here they coincide due to this particular 
property of exponential growth.

Let $N_k(X)$ denote the number of elements of the uniformly distributed set
$X$ that are in $\R^n \times [2^{-(k+1)},2^{-k}]$.
Then, $\vol\,(\R^n \times [2^{-(k+1)},2^{-k}] \cap H_\varepsilon(\Lbd))$ is 
comparable with $N_k(X)$ and hence
$$
\h(\Lbd)=\limsup_{k \to \infty} \frac{\log N_k(X)}{k \log 2}.
$$

Next, we compare $N_k(\Lbd)$ with $N_k(X)$. For every point 
$x_k(m_1,\ldots, m_n) \in X$, there is some $I_k(m'_1,\ldots, m'_n)$ 
intersecting $\Lbd$ with $|m_i-m'_i| \leq 1$ $(i=1,\ldots,n)$.
Conversely, for every square $I_k(m_1,\ldots, m_n)$ intersecting $\Lbd$, 
there is some point $x_k(m'_1,\ldots, m'_n) \in X$ with 
$|m_i-m'_i| \leq 1$ $(i=1,\ldots,n)$.
This shows that
$$
3^{-n} N_k(\Lbd) \leq N_k(X) \leq 3^n N_k(\Lbd),
$$
and hence 
$$
\h(\Lbd)=\overline {{\rm dim}}_B (\Lbd).
$$

Finally, in the case that $\Lbd=\Lbd(G)$ for some non-elementary 
Kleinian group $G$, the additional inequality stated in the theorem follows 
from the theorem of Bishop and Jones \cite{bj97} asserting that 
$\delta(G) = \dim_H (\Lbd_c(G))$.
\end{proof}

\begin{remark*}
Another exponent of similar flavour related to the upper box-counting dimension
has been introduced by Bishop \cite{bis96}. For any closed set 
$\Lambda \subset \S^n$, which we may assume to be contained in $\R^n$,
consider a Whitney decomposition $\mathcal Q$ of the complement 
$\R^n \setminus \Lambda$, which consists of dyadic cubes $\{Q_i\}_{i \in \N}$ 
of disjoint interior whose diameters are comparable to the distance to $\Lambda$.
Then the Whitney exponent $\Delta_{\rm Whit}(\mathcal Q)$ is defined to be the 
critical exponent 
of the series $\sum_{i \in \N} \diam(Q_i)^s$, and is
independent of the choice of the Whitney decomposition $\mathcal Q$.
Then $\Delta_{\rm Whit}(\mathcal Q) \leq {\overline{\rm dim}}_B (\Lambda)$ 
in general and 
$\Delta_{\rm Whit}(\mathcal Q) = {\overline{\rm dim}}_B (\Lambda)$ if $\Lambda$ 
has null $n$-dimensional measure. 
This is a generalisation of a similar result due to Tricot \cite{tri81}
for the $1$-dimensional case.
\end{remark*}

The volume entropy on a horoball can be considered similarly to the 
convex core entropy as follows.

\begin{lemma}\label{horoballentropy}
Let $D$ be some horoball in $\B^{n+1}$ 
and $B_R(z)$ the ball of radius $R>0$ centred at some $z \in \B^{n+1}$. Then 
$$
\lim_{R \to \infty} 
\frac{\log \vol\,(B_R(z) \cap D)}{R}=\frac{n}{2}.
$$
\end{lemma}

\begin{proof}
We will show the statement in the upper half-space model 
$$
\H^{n+1}=\{(x,y) \mid x \in \R^n,\ y>0 \}
$$ 
of hyperbolic space. We may assume that $D=\{y \geq 1\}$ and
$z=(0,1)$. Then the intersection of $B_z(R)$ with the boundary
$\partial D=\{y = 1\}$ is the ball of Euclidean radius 
$\sqrt{2(\cosh R-1)}$. From this, we see that the volume of
the $c$-neighbourhood of $\partial D$ in $B_z(R)$ for some $c>0$ is 
comparable with $\exp(nR/2)$, and hence
$$
\liminf_{R \to \infty} \frac{\log \vol\,(B_R(z) \cap D)}{R} \geq \frac{n}{2}.
$$

On the other hand, $B_R(z) \cap D$ is contained in a region
$$
E=\{(x,y) \in \H^{n+1} \mid \vert x \vert \leq 
\sqrt{(\cosh R-1)/2}\,(y+1),\ 1 \leq y \leq e^R\},
$$
and $\vol\,(E)$ is bounded by
$$
\left(\frac{\cosh R-1}{2}\right)^{\frac{n}{2}} 
\int_1^{e^R}\frac{(y+1)^n}{y^{n+1}} dy
\leq R \exp(nR/2).
$$
This gives 
$$
\limsup_{R \to \infty} \frac{\log \vol\,(B_R(z) \cap D)}{R} \leq \frac{n}{2},
$$
and hence the assertion follows.
\end{proof}

\begin{remark*}
The volume of $B_R(z) \cap D$ is explicitly given by the definite integral
$$
\frac{\pi^{\frac{n}{2}}}{\Gamma(\frac{n}{2}+1)}\,\int_1^{e^R} 
\frac{(-y^2+2y \cosh R-1)^{\frac{n}{2}}}{y^{n+1}} dy.
$$
In the case where $n=1$, this is equal to 
\begin{eqnarray*}
& &2\cosh R\,\left(\frac{\pi}{2}-\arctan \sqrt{\frac{\cosh R-1}{2}}\right)
+2\sqrt{2(\cosh R-1)}\\
& &-2\left(\frac{\pi}{2}+\arctan \sqrt{\frac{\cosh R-1}{2}}\right).
\end{eqnarray*}
In the general case, this can be represented by using hypergeometric series.
\end{remark*}

As a direct consequence of Lemma~\ref{horoballentropy} we obtain the following.

\begin{theorem}
\label{lwr-bd-cx-hull-entr}
If the convex hull $H(\Lambda)$ of a closed set $\Lambda \subset \S^n$
contains a horoball, then $h_c(\Lambda) \geq n/2$. In particular, for
any Kleinian group $G$ acting on $\B^{n+1}$ with the property
that the convex hull of its limit set contains a horoball, we have 
$$
h_c(\Lambda(G)) \geq \frac{n}{2}.
$$
\end{theorem}

\begin{remark*}
The above condition on the convex hull is trivially satisfied for 
non-elementary Kleinian groups $G$ with parabolic elements of maximal rank $n$. 
In that case, it is well-known \cite{bea68} that the critical exponent 
$\delta(G)$ is strictly greater than $n/2$, and we recover the statement of 
Theorem~\ref{lwr-bd-cx-hull-entr} since $\h(\Lambda(G))\geq \delta(G) > n/2$
by Theorem~\ref{basicinequality}. 
\end{remark*}

Note that there is a canonical way of putting a uniformly distributed set 
on a horoball $D$. See Cannon and Cooper \cite{cc92} for literature.
For instance, in the horoball $D=\{(x,y) \mid x \in \R^n,\ y \geq 1\}$ 
of the upper half-space model $\H^{n+1}$ as above, we can define
$$
X_D=\bigcup_{k \in \N \cup \{0\}}((2^k\Z)^n,2^k)
$$
to be a uniformly distributed set.

Using this uniformly distributed set on a horoball, it is not too 
difficult to show that for a geometrically finite Kleinian group
the convex core entropy coincides with the critical exponent.
However, this can be immediately deduced from Theorem~\ref{basicinequality}
and Stratmann and Urba{\'n}ski's result \cite{stur96} that the
Hausdorff, packing and box-counting dimensions of the limit set 
of a geometrically finite Kleinian group are coincident and agree with the 
critical exponent.

\begin{proposition}\label{geomfiniteequal}
For a non-elementary geometrically finite Kleinian group $G$,
we have that
$$
\delta(G)={\rm dim}_H(\Lambda(G))={\rm dim}_B(\Lambda(G))=\h(\Lbd(G)).
$$ 
\end{proposition}

For a Kleinian group $G$ acting on $3$-dimensional hyperbolic space $\B^3$,
an extension of this proposition was obtained by Bishop \cite{bis97} and
also in \cite{bj97}. A Kleinian group $G$ is called {\it analytically finite} if
$(\S^2-\Lambda(G))/G$ consists of at most a finite number of compact Riemann 
surfaces possibly with finitely many punctures. 
Remark that this class includes any group $G$ of the first kind 
(for which $\Lambda(G)=\S^2$).
Geometrically finite groups are analytically finite and so are finitely 
generated groups acting on $\B^3$ by the Ahlfors finiteness theorem.
It was proved in \cite{bis97} that, if a non-elementary Kleinian group $G$ is analytically
finite and $\Lambda(G) \neq \S^2$, then 
$\delta(G)$ coincides with $\Delta_{\rm Whit}(\mathcal Q)$
for a Whitney decomposition $\mathcal Q$ of the complement of $\Lambda(G)$.
Note that the aforementioned result in \cite{tri81} and \cite{bis96} asserts that
this is equal to $\udB(\Lbd(G))$ if $\Lambda(G)$ has zero area.

On the other hand, the rather deep result of Bishop and Jones \cite{bj97} shows
that, if $G$ is analytically finite and geometrically infinite and if $\Lbd(G)$
has zero area, then $\delta(G)=2$. This in particular implies that,
whenever a Kleinian group $G$ is 
analytically finite, the box-counting dimension of its limit set exists
and coincides with the Hausdorff dimension $\dim_H(\Lambda(G))$.
Also by the solution of the Ahlfors conjecture (which follows from 
the Tameness Conjecture proven by Agol~\cite{ag04} and independently by 
Calegari and Gabai~\cite{cg06}),
every finitely generated Kleinian group $G$ has the limit set of zero area if $\Lambda(G) \neq \S^2$
and $\delta(G)=2$ if $\Lambda(G) = \S^2$.
In summary, we have the following theorem obtained by piecing together 
results from \cite{bis96}, \cite{bis97}, \cite{stur96}, 
\cite{bj97}, \cite{ag04} and \cite{cg06}.

\begin{theorem}
For a non-elementary analytically finite Kleinian group $G$ acting on $\B^3$, 
we have that
$$
\dim_H(\Lambda(G))=\dim_B(\Lambda(G))=\h(\Lambda(G)).
$$
In addition, if $G$ is finitely generated or if the $2$-dimensional measure 
of $\Lambda(G)$ is zero, then $\delta(G)$ coincides with these quantities.
\end{theorem} 

Finally, in this section, we introduce a class of convex hulls $H(\Lambda)$ 
in hyperbolic space $\B^{n+1}$
such that the convex core entropy $h_c(\Lambda)$ is strictly less than $n$.

\begin{definition}
We say that a convex hull $H(\Lambda)$ in $\B^{n+1}$ is 
{\it $\ell$-tight} (or just tight in brief)
if every point $z \in H(\Lambda)$ is within distance $\ell \geq 0$ from
the boundary $\partial H(\Lambda)$. Also $H(\Lambda)$ is
{\em weakly $\ell$-tight} if there exist mutually disjoint horoballs 
$\{D_p\}_{p \in \Phi}$ with the set of tangency points $\Phi \subset \S^n$
such that every point $z \in H(\Lambda)\setminus \bigcup_{p \in \Phi} D_p$ 
is within distance $\ell$ from $\partial H(\Lambda)$.
\end{definition}

As first natural examples of Kleinian groups with the weak tightness
property we have all non-trivial normal subgroups of
geometrically finite groups of the second kind.
For a Kleinian group $G$ acting on $\B^{n+1}$, 
generalising an argument of Tukia \cite{tuk84} for
geometrically finite groups, it was proven in \cite{mat00} that,
if the convex hull $H(\Lambda(G))$ of the limit set of $G$ is 
weakly tight, then the Hausdorff dimension $\dim_H(\Lambda(G))$ 
is strictly less than $n$. In \cite{str06}, $G$ is called 
{\it geometrically tight} if $H(\Lambda(G))$ is weakly tight.
By essentially the same arguments, 
we can show that the convex core entropy $h_c(\Lambda(G))$ is
strictly less than $n$ also in this case. 
Actually, this fact is independent of group actions and valid in general 
for an arbitrary convex hull with the tightness property.

\begin{theorem}\label{bdneighbouring}
If a convex hull $H(\Lambda)$ in $\B^{n+1}$ is weakly $\ell$-tight,
then 
$$
h_c(\Lambda) \leq s(n,\ell) <n,
$$
where $s(n,\ell)$ is some constant depending only on $n$ and $\ell$. 
In particular, if a Kleinian group $G$ is geometrically tight,
then $h_c(\Lambda(G))<n$.
\end{theorem}

\begin{proof}
By transferring $H(\Lbd)$ and $\Lbd$ to the upper-half plane model $\H^{n+1}$ 
of hyperbolic space together with its boundary $\R^n \cup \{\infty\}$,
we can assume that $\Lambda$ is a bounded subset of $\R^n$. 
For a Euclidean cube $Q$ in $\R^n$, we denote by $\overline Q$
the Euclidean cube in $\overline{\H^{n+1}}$ having $Q$ as the lower face.
Also $d(Q)$ denotes its side length and $x(Q) \in \H^{n+1}$ the centre of 
the upper face of $\overline Q$.

Assume that $H(\Lambda)$ is weakly $\ell$-tight.
It is shown in \cite{mat00} that we can choose the following family 
$\mathcal L_m=\{Q_{m,i}\}_{i \in \N}$ of cubes $Q_m=Q_{m,i}$ for each $m \in \N$. 
For every fixed $m$, all cubes $Q_m \in \mathcal L_m$ have mutually disjoint 
interior and intersect $\Lambda$. Also the union $\bigcup Q_m$ taken over all 
$Q_m \in \mathcal L_m$ covers all of $\Lambda$ except the set $\Phi$ of 
tangency points of the horoballs $\{D_p\}_{p \in \Phi}$ which appear in 
the definition of weak tightness for $H(\Lambda)$. 
Each cube $Q_m \in \mathcal L_m$ contains smaller cubes $Q_{m+1}$ in 
$\mathcal L_{m+1}$ that satisfy the following properties. Here 
we denote the family of these cubes $Q_{m+1}$ contained in a given $Q_m$ by 
$\mathcal L_{m+1}(Q_m)$. 
\begin{itemize}
\item 
For every point $z$ on the boundary $\partial H(\Lambda)$,
there is a upper centre $x(Q_m)$ of some $Q_m \in \mathcal L_m$, $m \in \N$, 
within a uniformly bounded distance of $z$.
\item
There exist positive constants $s=s(n,\ell)<n$ and $c<1$ depending only on $n$ 
and $\ell$ such that 
$$
\sum_{Q_{m+1} \in \mathcal L_{m+1}(Q_m)} d(Q_{m+1})^s \leq c d(Q_m)^s
$$
for each $Q_m \in \mathcal L_{m}$ and every $m \in \N$.
\end{itemize}
Note that the first property above is not explicitly written down in 
\cite{mat00} but it is clear from the construction.

We choose a point $x_m \in H(\Lambda)$ for each $Q_m$
that is closest to $x(Q_m)$.
Let $X'$ be the union of all such points $x_m$ taken over 
all $Q_m \in \mathcal L_m$ for all $m \in \N$. 
Then $X'$ is uniformly distributed everywhere in some neighbourhood of 
$\partial H(\Lbd)$ apart from the union of $\{D_p\}_{p \in \Phi}$. 
For each horoball $D=D_p$, we put a canonical set of points $X_{D'}$ in a 
uniformly distributed manner into $D'=D \cap H(\Lbd)$. 
Then the union of $X'$ and $X_{D'}$ for all horoballs defines 
a uniformly distributed set $X$ in $H(\Lambda)$ due to its weak tightness.

In each $X_{D'}$, we take the nearest point $x_{D'}$ to a given base point 
$o \in \H^{n+1}$.
By elementary hyperbolic geometry, we see that there is a constant $a>0$ 
independent of $D'$ such that
$$
d(x,o)+a \geq d(x,x_{D'})+d(x_{D'},o)
$$
is satisfied for every $x \in X_{D'}$. This yields
\begin{eqnarray*}
P^s(X_{D'},o) &=&\sum_{x \in X_{D'}}e^{-sd(x,o)}\\
&\leq& \sum_{x \in X_{D'}}e^{-s(d(x,x_{D'})+d(x_{D'},o)-a)}
= e^{sa}e^{-sd(x_{D'},o)}P^s(X_{D'},x_{D'}).
\end{eqnarray*}
For the sake of simplicity, we may assume that $X'$ contains each $x_{D'}$
by enlarging $X'$. Then we have
$$
P^s(X,o) \leq e^{sa} P^s(X',o) \cdot P^s(X_{D'},x_{D'}),
$$
where $D$ can be chosen arbitrarily from $\{D_p\}$.
Using this inequality and Lemma \ref{horoballentropy}, we have 
$h_c(\Lambda)=\Delta(X) \leq \max\{\Delta(X'),n/2\}$, 
where $\Delta(X')$ is the critical exponent defined for $X'$ instead of $X$.

In order to estimate $\Delta(X')$, we use the second property of 
$\mathcal L_m=\{Q_m\}$ as above.
It is well-known that the hyperbolic distance $d(o,x_m)$ can be estimated by
using the Euclidean distance from $x(Q_m)$ to the boundary $\R^n$. 
In particular, there is a uniform constant $C>0$ such that 
$e^{-d(o,x_m)} \leq C d(Q_m)$.
Then, for $s=s(n,\ell)$, we have
$$
P^s(X',o) \leq C^s \sum_{m \in \N} \sum_{Q_m \in \mathcal L_m} d(Q_m)^s
\leq C^s \sum_{m \in \N} c^m<\infty,
$$
which shows that $\Delta(X') \leq s(n,\ell) <n$.
\end{proof}

In the case where $H(\Lambda)$ is $\ell$-tight without negligible horoballs,
we can obtain a better estimate of the convex core entropy $h_c(\Lambda)$
in terms of $\ell$. This is done by using the coincidence of $h_c(\Lambda)$ with
the upper box-counting dimension $\overline{\dim}_B (\Lambda)$ and estimating 
the porosity constant for $\Lambda$ in terms of $\ell$. Then the upper estimate 
of $\overline{\dim}_B (\Lambda)$ in terms of the porosity constant yields 
the result. One may think that the above theorem can also be obtained in this 
way, by using the relation between $\overline{\dim}_B (\Lambda)$ and the 
so-called mean porosity introduced in \cite{kr97}, but it seems that sets 
$\Lambda$ for which $H(\Lambda)$ is weakly tight are not necessarily mean porous 
as in the definition, so the known estimate cannot be applied directly.

\begin{theorem}\label{porosityestimate}
If the convex hull $H(\Lambda) \subset \B^{n+1}$ of a closed set 
$\Lambda \subset \S^n$ is $\ell$-tight, then
$$
h_c(\Lambda)=\overline{\dim}_B (\Lambda) \leq n-C \, e^{-2n\ell},
$$
where $C>0$ is a constant depending only on $n$.
\end{theorem}

\begin{proof}
We again use the upper half-space model $\H^{n+1}$ and
assume that
$\Lambda$ is a closed set of $\R^n \cup \{\infty\}$ containing $\infty$
and set $\Lambda'=\Lambda-\{\infty\} \subset \R^n$.
We will show that, if $H(\Lambda)$ is $\ell$-tight,
then $\Lambda'$ is $c$-porous for $c=2e^{-2\ell}$. 
Then by the well-known estimate of the upper box-counting dimension
$$
\overline{\dim}_B (\Lbd') \leq n-C \, c^n
$$
and by Theorem \ref{basicinequality}, we have the required inequality.

Suppose to the contrary that $\Lbd'$ is not $c$-porous for $c=2e^{-2\ell}$. 
Then there are $x \in \Lbd'$ and $r>0$ such that
every point $y$ in the Euclidean ball $D_r(x) \subset \R^n$ satisfies 
$d\,_{\R^n}(y,\Lbd') < cr$.
It is easy to see that $D_r(x) \times [cr,\infty) \subset \H^{n+1}$ is contained 
in the convex hull $H(\Lambda)$, and in particular the Euclidean ball 
$\tilde D_r(z) \subset \H^{n+1}$ of centre $z=(x, cr+r)$ is in $H(\Lambda)$.
Since the hyperbolic radius of $\tilde D_r(z)$ is $\frac{1}{2} \log (1+2/c)$, 
the hyperbolic centre $z_0 \in H(\Lambda)$ of $\tilde D_r(z)$ satisfies 
$$
d(z_0, \partial H(\Lambda)) \geq \frac{1}{2} \log \left(1+\frac{2}{c}\right) > 
\frac{1}{2} \log \frac{2}{c}=\ell.
$$
This implies that $H(\Lambda)$ is not $\ell$-tight, which proves the claim.
\end{proof}

\section{Amenability of pants decompositions}
\label{amenab-pants-section}

We define amenability of graphs and then apply this concept to a certain type
of decomposition of a hyperbolic surface. A graph ${\mathcal G}=(V,E)$ consists 
of the set $V$ of vertices and the set $E$ of edges that connect vertices in $V$. 
A subgraph $K=(V',E')$ of ${\mathcal G}$ is a graph with $V' \subset V$ and 
$E' \subset E$ such that every edge in $E'$ connects vertices in $V'$.
The boundary $\partial K$ of a subgraph $K=(V',E')$ is the subset $V''$ of 
vertices in $V'$ that are connected to a vertex in $V \setminus V'$ by an 
edge in $E$.

\begin{definition}
\label{notions-amenability}
A connected graph ${\mathcal G}=(V,E)$ is called \emph{amenable} if
$$
\inf_K \frac{|\partial K|}{|K|} = 0,
$$
where the infimum is taken over all finite connected subgraphs $K$.
Moreover, ${\mathcal G}$ is called \emph{strongly amenable} if
$$
\lim_{m \to \infty} \inf_{K^{(m)}} \frac{|\partial K^{(m)}|}{|K^{(m)}|} = 0,
$$
where the infimum is taken over all finite connected subgraphs $K^{(m)}$
with $|K^{(m)}| \leq m$ and $v \in K^{(m)}$ for some fixed vertex $v \in V$.
Finally, ${\mathcal G}$ is defined to be \emph{uniformly strongly amenable} if
the above convergence as $m \to \infty$ is uniform independently of
the choice of $v \in V$.
\end{definition}

For more details on the notion of amenability see e.g. \cite{bro81}, 
\cite{zim84}. For amenability of graphs see e.g. \cite{woe00} and the appendix 
in \cite{hss00}. Note that a graph is assumed to be simple in some cases but our 
graph defined later can contain a loop edge at a vertex and double edges between 
two vertices. However, as far as our arguments are concerned, these edges 
do not cause problems. If necessary, we may add virtual vertices at the middle 
of all edges to make the graph simple.

\begin{remark*}
As a direct consequence of the definition, a graph is non-amenable 
if and only if
$$
\sup_K \frac{|K|}{|\partial K|} < \infty,
$$
where the supremum is again taken over all finite connected subgraphs $K$.
\end{remark*}

We say that a finitely generated group $G$ is (non-) amenable if the Cayley 
graph of $G$ with respect to some generating set $A$ is (non-) amenable.
This definition is independent of the choice of the generating set $A$.
Strong amenability and uniformly strong amenability of $G$ 
are defined in the same way.
However, since $\Gr$ is \emph{homogeneous} in the sense that
$G$ acts on the vertices of its Cayley graph transitively,
we see that these concepts of amenability are all the same for finitely 
generated groups.

A normal covering $M_2 \to M_1$ is called amenable if its covering 
transformation group, which is isomorphic to the quotient group 
$\pi_1(M_1)/\pi_1(M_2)$,
is amenable.

The following theorem is due to Brooks \cite{bro85}. Here, $\lambda_0$ denotes 
the bottom of the $L^2$-spectrum of a Riemannian manifold. 
Note that, for a covering $M_2 \to M_1$, the inequality 
$\lambda_0(M_1) \leq \lambda_0(M_2)$ is always satisfied.

\begin{theorem}[Brooks \cite{bro85}]
\label{brooks}
Let $M_1$ be a $(n+1)$-dimensional hyperbolic manifold given by a 
convex cocompact Kleinian group with $\lambda_0(M_1) < n^2/4$, 
and let $M_2$ be a normal cover of $M_1$. 
Then the covering $M_2 \to M_1$ is amenable if and only if 
$\lambda_0(M_2) = \lambda_0(M_1)$.
\end{theorem}

Recently, Stadlbauer~\cite{sta11} gave a new proof of Brooks' theorem
by developing further ideas of Kesten \cite{kes59} on random walks
on graphs, which allows one to drop the assumption $\lambda_0(M_1) < n^2/4$.

We also have the well-known correspondence between $\lambda_0$ and $\delta$
given by the following theorem.

\begin{theorem}[Elstrodt \cite{els73}, Patterson \cite{pat76}, 
Sullivan \cite{sul87}]
\label{els-pat-sul}
For any Kleinian group $G$ and $M_G=\B^{n+1}/G$ we have
$$
\lambda_0(M_G) = \left\{
\begin{array}{ll}
\frac{1}{4} \, n^2, & \mbox{if $\delta(G) \leq \frac{1}{2}\,n$} 
\vspace{2mm} \\
\delta(G)(n-\delta(G)), & \mbox{if $\delta(G) \geq \frac{1}{2}\,n$.} 
\end{array}
\right.
$$
\end{theorem}

Note that the less intricate implication in Brooks' theorem, namely that
amenabi\-li\-ty implies equality of critical exponents, has been recovered 
and even generalised using different methods by Roblin~\cite{rob05} and 
Sharp~\cite{sha07}. The above results enable us to prove the following theorem. 

\begin{theorem}\label{normalsubgroup}
If $N$ is a non-trivial normal subgroup of a convex cocompact Klein\-ian
group $G$, then 
$$
\dim_H(\Lbd(N)) = \h(\Lbd(N)).
$$
Furthermore, 
$$
G/N: \mbox{ non-amenable } \; \Longleftrightarrow \;\; 
\delta(N) < \dim_H(\Lbd(N)).
$$
\end{theorem}

\begin{proof}
We deduce the first statement from 
$\Lbd(N) = \Lbd(G)$ and $\dim_H(\Lbd(G)) = \delta(G) = \Delta(X)$, 
where $X$ is chosen to be some $G$-orbit in $H(\Lbd(G))$, 
together with Proposition~\ref{cc-entr-is-delta-X}.
The second statement follows from Brooks' theorem (Theorem~\ref{brooks}),
enhanced by Stadlbauer~\cite{sta11}. 
\end{proof}

As an illustration of the situation in Theorem~\ref{normalsubgroup} 
we give the following class of examples of
Kleinian groups whose critical exponent is strictly smaller than the 
Hausdorff dimension of their limit set (see also \cite{fs04}).

\begin{example}
\label{dim-gap-expl}
Let $G_0$ and $ G_1$ be two non-elementary convex cocompact
Kleinian groups acting on $\B^{n+1}$ with (open) fundamental domains 
$F_0$, $F_1$ respectively, such that 
$F_0^{{\bf c}}\, \cap F_1^{{\bf c}} = \emptyset$.  
For simplicity, we assume that $G_0$ is freely
generated by $g_1,\ldots, g_k$.
With $G = G_0 \, \ast \, G_1$,
let $\varphi : G \to G_1$ denote the canonical group
homomorphism, and define $N:= {\rm ker}(\varphi )$.  Clearly,
$$
N = \langle hg_ih^{-1}: i=1,...,k, h \in G_1 \rangle, 
$$
and $N$ is the normal subgroup of $G$ generated by $G_0$ in $G$.  
It follows that $G/N$ is isomorphic to $G_1$.  
Observe now that $G/N$ contains a free 
subgroup on two generators, and therefore $G/N$ is not amenable. Hence,
applying the theorem above, it follows that 
$\delta(N) < \dim_H(\Lbd(N))$.
\end{example}

In what follows, we restrict ourselves to 2-dimensional hyperbolic surfaces
and consider the amenability of objects which do not necessarily originate 
from normal coverings.

\begin{definition}
A pair of pants is a complete hyperbolic surface with geodesic boundary
whose interior is homeomorphic to the complement of three points in the
2-sphere. 
\end{definition}

Note that a pair of pants can contain cuspidal regions. Also, 
a connected hyperbolic surface with geodesic boundary is a pair of pants 
if and only if all of its simple closed geodesics are boundary components.
Thus, pairs of pants are exactly those 
hyperbolic surfaces with geodesic boundary which cannot be further simplified by 
cutting along simple closed geodesics.

The following definition contains an existence statement which is formulated
and proven in \cite{ar04} and \cite[Theorem 3.6.2]{hub06}.

\begin{definition}\label{pantsdecomposition}
Every connected, not simply connected hyperbolic surface $S$ has a 
\emph{pants decomposition} $\Pa$ obtained as follows. There exists a family of 
mutually disjoint simple closed geodesics 
$M$ on $S$ such that if $\overline{M}$ denotes the closure of $M$
in $S$ (viewed as a subset of $S$), then the closure of each connected 
component of $S \setminus \overline{M}$ is isometric to either
\begin{enumerate}
\item a pair of pants (thus the term pants decomposition),
\item a half annulus $\{z \in \C : 1 \leq |z| < R\}$ in 
$\{z \in \C : 1/R < |z| < R\}$, for some $1<R<\infty$, 
endowed with the corresponding hyperbolic metric, or
\item a closed half plane in $\B^2$ with the hyperbolic metric.
\end{enumerate}
$\Pa$ is the family of all pairs of pants (components of type (1)), which is in
the convex core $C(S)$ of $S$.
\end{definition}

For details on the pants decomposition of a hyperbolic surface we refer
the reader for instance to \cite[Sections 3.5 and 3.6]{hub06}.

\begin{remark*}
Pants decompositions are not unique as can be easily seen already for the simple
example of a compact surface of genus 2. 
\end{remark*}

\begin{definition}
For a pants decomposition $\Pa$ of the hyperbolic surface $S$,  
the \emph{associated graph} $\Gr = \Gr(\Pa)$ is defined to be
the graph whose vertices are given by the pairs of pants in $\Pa$ 
and which has edges between pairs of pants whenever these are adjacent via 
some boundary simple closed geodesic. We endow the graph $\Gr(\Pa)$ with the
usual metric $d_\Pa$ in which edges have length 1. 
\end{definition}

\begin{remark*}
A pair of pants can be adjacent to itself, and two pairs of pants can be 
adjacent via two distinct simple closed geodesics. In other words, 
the graph $\Gr(\Pa)$ can contain loop edges at some vertices and double edges 
between vertices.
\end{remark*}

\begin{definition}
We say that 
$S$ has \emph{strongly bounded geometry} if 
$S$ admits a pants decomposition $\Pa$
such that there are constants 
$C \geq c > 0$ with the property that the lengths of geodesic boundary 
components of each pair of pants possibly containing a cusp
are bounded from above by $C$ and from below by $c$. We call such $\Pa$
\emph{uniform pants decomposition}.
\end{definition}

\begin{theorem}
\label{pantsgraph-qi-convexcore}
Let $S$ be a hyperbolic surface with strongly bounded geometry and $\Pa$
a uniform pants decomposition of $S$.
If $\Gr$ is the graph associated to $\Pa$, then $\Gr$ is 
quasi-isometric to the convex core $C(S)$ 
without the union of all canonical cusp neighbourhoods. 
In particular, if $\Gr$ and $\Gr'$ are graphs associated to 
different uniform pants decompositions of $S$, 
then $\Gr$ and $\Gr'$ are quasi-isometric.
\end{theorem}

\begin{proof}
The second statement follows directly from the first. The proof of the first
statement is a standard argument which we give for the sake of completeness.
Let $x$ and $y$ be points in $C(S)$ without the union of all 
canonical cusp neighbourhoods, and let $\gamma$ be the shortest geodesic path 
connecting them. Then $\gamma$ decomposes into geodesic paths 
$\gamma_i = \gamma \cap P_i$, $i\in\{1,\ldots,n\}$, where the $P_i$ are 
the pairs of pants from $\Pa$ met by $\gamma$. 
Since $\Pa$ is a uniform pants decomposition, it is not difficult to see that 
there exist constants $C' \geq c' > 0$ depending only on $\Pa$ so that 
$$
c' \leq \ell(\gamma_i) \leq C', \;\; i \in \{1,\ldots,n\}.
$$
Since $d(x,y) = \ell(\gamma) = \sum_{i=1}^n \ell(\gamma_i)$ we thus obtain
$$
c' d_\Pa(P_1,P_n) -2c' \leq d(x,y) \leq C'n.
$$
Now let $n_{\min}$ be the minimal number of pants in $\Pa$ that \emph{must} 
be crossed by some geodesic segment of length $d(x,y)$.
Then we clearly have
$n_{\min} \geq d(x,y)/C'$ and thus
$$
n \leq \frac{d(x,y)}{c'} \leq \frac{C'}{c'} n_{\min} 
\leq \frac{C'}{c'} d_\Pa(P_1,P_n).
$$ 
Therefore,
$$
d(x,y) \leq C'n \leq \frac{(C')^2}{c'} d_\Pa(P_1,P_n),
$$
which finishes the proof.
\end{proof}

By using \cite[Theorem~4.7]{woe00}, which essentially
ensures the invariance of amenability under quasi-isometries, we obtain the
following corollary as a direct consequence of 
Theorem~\ref{pantsgraph-qi-convexcore}.

\begin{corollary}
If $S$ is a hyperbolic surface with strongly bounded geometry, 
and if $\Gr$ and $\Gr'$ are graphs associated to different uniform pants 
decompositions of $S$, then 
$$
\Gr: \mbox{ amenable } \; \Longleftrightarrow \;\; \Gr': \mbox{ amenable.}
$$
\end{corollary}

\begin{remark*}
\label{pantsgraph-unifdistrset}
Given a hyperbolic surface $S=\B/G$ with strongly bounded geometry, 
and a uniform pants decomposition $\Pa$ of $S$, 
one can produce a uniformly distributed set $X$ in $H(\Lambda(G))$ 
in the following way: First, choose a point in each pair of pants of $\Pa$ 
and take the lifts of all such points to $H(\Lambda(G))$;
secondly, in each horoball $D$ that is a lift of a cusp neighbourhood of $S$,
we take the set of points $X_D$ canonically as in Section \ref{3}.
\end{remark*}

\begin{remark*}
If $N$ is a non-trivial normal subgroup of a convex cocompact Fuchsian
group $G$ acting on $\B^2$ so that the normal cover $S = \B^2/N$ preserves
the structure of a pants decomposition $\Pa$ of $\B^2/G$, 
then $S$ has strongly bounded geometry, and, 
by Theorem~\ref{pantsgraph-qi-convexcore}, 
the graph $\Gr(\Pa)$ associated to $\Pa$ is 
quasi-isometric to the Cayley graph of $G/N$.
\end{remark*}

\section{Amenability and the dimension gap}

In this section, we continue to investigate hyperbolic surfaces $S=\B^{n+1}/G$
with strongly bounded geometry and prove the following two theorems. 
They provide partial information on the relationship between amenability of 
the graph $\Gr=\Gr(\Pa)$ associated to a uniform pants decomposition $\Pa$
and gaps between the three quantities $\delta(G)$, $\dim_H(\Lbd(G))$ and 
$h_c(\Lbd(G))$.

\begin{theorem}
\label{amen-implies-equal}
Let $S$ be a hyperbolic surface given by the Fuchsian group $G$. Assume
$S$ is of infinite type with strongly bounded geometry. 
Let $\Gr$ be the graph associated to a uniform pants decomposition $\Pa$ of $S$. 
If $\Gr$ is uniformly strongly amenable, then all $\delta(G)$, $\dim_H(\Lbd(G))$ 
and $h_c(\Lbd(G))$ are either equal to $1$ or less than $1$ at the same time.
\end{theorem}

\begin{theorem}
\label{non-amen-equiv-delta<1}
Let $S$ be a hyperbolic surface given by the Fuchsian group $G$. Assume
$S$ is of infinite type with strongly bounded geometry.  
Let $\Gr$ be the graph associated to a uniform pants decomposition $\Pa$ of $S$. 
If $\Gr$ is non-amenable, then $\delta(G) < 1$.
\end{theorem}

In order to prove Theorems \ref{amen-implies-equal} and 
\ref{non-amen-equiv-delta<1}, we shall need some prerequisites.

\begin{definition}
\label{cheeger-ct}
The \emph{isoperimetric constant} or \emph{Cheeger constant} of a 
hyperbolic surface $S$ is defined as
$$
h(S) = \sup_W \frac{A(W)}{\ell(\partial W)},
$$
where the supremum is taken over all relatively compact domains 
$W \subset S$ with smooth boundary $\partial W$. 
Here, $A(W)$ denotes the hyperbolic area of $A$ and 
$\ell(\partial W)$ the hyperbolic length of the boundary. 
\end{definition}

\begin{remark*}
\label{cheeger-thm}
Cheeger's constant satisfies $h(S) \geq 1$ for any hyperbolic surface $S$.
See \cite{am99} for a proof not relying on Theorems~\ref{els-pat-sul} 
and \ref{estimates}.
\end{remark*}

Fundamental estimates of $\lambda_0(S)$ in terms of $h(S)$ are given as follows
by Cheeger \cite{che70} and Buser \cite{bus82}.

\begin{theorem}\label{estimates}
For every hyperbolic surface $S$ we have
$$
\frac{1}{4 \, h(S)^2} \leq \lambda_0(S) \leq \frac{B}{h(S)}
$$
for some universal constant $B>0$.
\end{theorem}

\begin{definition}
A \emph{geodesic domain} $W^*$ in the hyperbolic surface $S$ is a 
domain of finite area so that $\partial W^*$ consists of finitely many 
simple closed geodesics; define the isoperimetric constant given by such
domains as
$$
h^*(S) = \sup_{W^*} \frac{A(W^*)}{\ell(\partial W^*)},
$$
where the supremum is taken over all geodesic domains $W^*$ in $S$.
\end{definition}

By definition, a geodesic domain $W^*$ in $S$ does not contain any funnels 
(i.e. half annuli in (2) of Definition \ref{pantsdecomposition})
of $S$. Also, geodesic domains are not necessarily relatively compact, 
as they may contain cusps. Nonetheless, we have the following estimate, 
which is essentially a direct consequence of the definitions and \cite{fr90} and
\cite[Theorem~7]{mat05}.

\begin{proposition}
\label{matsuzaki}
Let $S$ be a hyperbolic surface of infinite type. 
Then,
$$
h^*(S) \leq h(S) \leq h^*(S) + 1.
$$
\end{proposition}

\begin{definition}
Let $S$ be a hyperbolic surface of infinite type with strongly bounded geometry, 
and $\Pa$ some uniform pants decomposition of $S$. Set
$$
h^{\Pa}(S) = \sup_{W^\sharp} \frac{A(W^\sharp)}{\ell(\partial W^\sharp)},
$$
where the supremum is taken over all finite connected unions $W^\sharp$ of pants  
from the pants decomposition $\Pa$. 
\end{definition}

\begin{lemma}
\label{h*-hP}
Let $S$ be a hyperbolic surface of infinite type with strongly bounded geometry, 
and $\Pa$ some uniform pants decomposition of $S$. Then there is a 
constant $b = b(\Pa) > 0$ so that
$$
h^*(S) \leq b \: h^{\Pa}(S).
$$
\end{lemma}

\begin{proof}
Take an arbitrary geodesic domain $W^*$ in $S$.
For this $W^*$, let $W^\sharp$ be the finite connected union of pants from 
$\Pa$ that consists of all pants having non-empty intersection with $W^*$.
Since $W^* \subset W^\sharp$, we have $A(W^*) \leq A(W^\sharp)$.
On the other hand, we compare $\ell(\partial W^*)$ with $\ell(\partial W^\sharp)$
by considering the number of pants in $W^\sharp$ intersecting $\partial W^*$.
Let $n \geq 1$ be this number. There is a constant $a=a(\Pa)>0$ such that
every simple closed geodesic in $S$ has at least one piece of length greater
than or equal to $a$ in common with any pair of pants from $\Pa$ it
intersects.
It follows that $\ell(\partial W^*) \geq an$. 
To each pair of pants in $W^\sharp$ intersecting $\partial W^*$,
there correspond at most two boundary components of $\partial W^\sharp$,
and the length of each boundary component is bounded from above by $C=C(\Pa)$,
which is the constant in the strongly bounded geometry condition.
Hence $\ell(\partial W^\sharp) \leq 2nC$.
Therefore we have
$$
\frac{A(W^*)}{\ell(\partial W^*)} \leq \frac{A(W^\sharp)}{an} \leq 
\frac{2C}{a} \frac{A(W^\sharp)}{\ell(\partial W^\sharp)}.
$$
By setting $b=2C/a$ and taking the supremum over all such $W^*$ and 
$W^\sharp$, we obtain the required inequality.
\end{proof}

We are now in the position to prove Theorems~\ref{amen-implies-equal} 
and \ref{non-amen-equiv-delta<1}.

\begin{proof}[Proof of Theorem~\ref{amen-implies-equal}]
We consider the case where the convex hull $H(\Lbd(G))$ of the limit set 
of $G$ is weakly tight.
This means that there exists a constant $\ell \geq 0$ 
such that every point $z$ in the
convex core $C(S)$ except in the canonical cusp neighbourhoods is within
distance $\ell$ from the boundary $\partial C(S)$.
If this is satisfied, then Theorem~\ref{bdneighbouring} asserts that 
$h_c(\Lbd(G))<1$. Since $\delta(G) \leq \dim_H(\Lbd(G)) \leq h_c(\Lbd(G))$ by 
Theorem~\ref{basicinequality}, these quantities are all less than $1$ 
in this case.

Conversely, assume that the above condition is not satisfied.
We will prove that $\delta(G)=1$ in this case. Then we also have
$\dim_H(\Lbd(G))=h_c(\Lbd(G))=1$ and the proof is complete.
By our assumption, we have a sequence of points $z_i$ in $C(S)$ without 
the union of canonical cusp neighbourhoods such that the distance from 
$z_i$ to $\partial C(S)$ tends to $\infty$ as $i \to \infty$.

Let $v_i$ be the vertex of the graph $\Gr$ corresponding to the pair of pants 
that contains $z_i$. Uniformly strong amenability of $\Gr$ implies that, 
for any $\varepsilon >0$, there exist a positive integer $m$ and, for each $i$,
a finite connected subgraph $K^{(m)}_i$ with $v_i \in K^{(m)}_i$ and
$|K^{(m)}_i| \leq m$, such that 
$$
\frac{|\partial K^{(m)}_i|}{|K^{(m)}_i|}<\varepsilon.
$$
Let $W_i^\sharp$ be a finite connected union of the pants from $\Pa$
corresponding to the subgraph $K^{(m)}_i$. Note that $W_i^\sharp$ contains $z_i$
and $W_i^\sharp$ consists of at most $m$ pants from $\Pa$. 
Uniformity of the pants decomposition $\Pa$ implies that each pair of pants 
in $\Pa$ has a bounded diameter except for canonical cusp neighbourhoods. 
Hence the diameters of the $W_i^\sharp$ without canonical cusp neighbourhoods 
are uniformly bounded independently of $i$.
Since the distance from $z_i$ to $\partial C(S)$ tends to $\infty$ as 
$i \to \infty$, the subsurfaces $W_i^\sharp$ are contained in the interior 
of $C(S)$ for all sufficiently large $i$.

We express $A(W_i^\sharp)$ and $\ell(\partial W_i^\sharp)$ by using 
$|K^{(m)}_i|$ and $|\partial K^{(m)}_i|$ for a sufficiently large $i$. 
It is easy to see that $A(W_i^\sharp)=2\pi |K^{(m)}_i|$.
By the fact that $W_i^\sharp$ is contained in the interior of $C(S)$,
the pairs of pants in $W_i^\sharp$ which correspond to elements of 
$\partial K^{(m)}_i$ are the only ones contributing to $\partial W_i^\sharp$, 
and have at most two boundary curves in $\partial W_i^\sharp$, that is,
the number of the boundary components of $\partial W_i^\sharp$ is 
bounded by $2|\partial K^{(m)}_i|$. Hence
$$
\frac{A(W_i^\sharp)}{\ell(\partial W_i^\sharp)}
\geq \frac{2\pi |K^{(m)}_i|}{2C |\partial K^{(m)}_i|}
> \frac{\pi}{C \varepsilon}.
$$
However, since $\varepsilon$ can be taken arbitrarily small, 
this estimate actually implies that $h(S)=\infty$. 
Then, by Theorems~\ref{els-pat-sul} and \ref{estimates}, 
we conclude that $\delta(G)=1$.
\end{proof}

\begin{proof}[Proof of Theorem~\ref{non-amen-equiv-delta<1}]
For the uniform pants decomposition $\Pa$ of $S$,
we consider some finite connected union $W^\sharp$ of pants 
from $\Pa$. Also,
for the graph $\Gr$ associated to $\Pa$, let $K$ be the subgraph of $\Gr$ 
associated to $W^\sharp$. 
Note that $K$ is connected and, by definition of 
$\Gr$, no pairs of pants are discarded while constructing $K$ from 
$W^\sharp$. 
The area $A(W^\sharp)$ is given by $2\pi |K|$.

The pairs of pants in $W^\sharp$ which correspond to elements of 
$\partial K$ have at least one 
boundary curve in 
$\partial W^\sharp$. Also, note that pairs of pants that produce
loop edges at vertices in $K$ do not contribute to 
$\partial W^\sharp$, for the corresponding vertices of $K$ cannot be in
$\partial K$. 
Therefore, the fact that $\Pa$ is a uniform pants decomposition implies 
$$
\ell(\partial W^\sharp) \geq c \, |\partial K|.
$$
But $\Gr$ was assumed to be non-amenable, and we conclude that
$$
h^{\Pa}(S) \leq \frac{2\pi}{c} \sup_K \frac{|K|}{|\partial K|} < \infty,
$$
where the supremum is taken over all finite connected subgraphs $K$ of
$\Gr$. By Proposition~\ref{matsuzaki} and Lemma~\ref{h*-hP} 
it now follows that
$$
h(S) \leq h^*(S) + 1 \leq b \: h^{\Pa}(S) + 1 \leq
\frac{2\pi b}{c} \sup_K \frac{|K|}{|\partial K|} + 1 < \infty.
$$
Finally, Theorem~\ref{estimates} implies that 
$\lambda_0(S) > 0$. Employing Theorem~\ref{els-pat-sul}, this
finishes the proof.
\end{proof}

We have the following corollary to both Theorems~\ref{amen-implies-equal} 
and \ref{non-amen-equiv-delta<1}.

\begin{corollary}\label{anyway-true}
Let $S$ be a hyperbolic surface given by the Fuchsian group $G$. Assume that
$S$ is of infinite type with strongly bounded geometry.  
Let $\Gr$ be the graph associated to some uniform pants decomposition 
$\Pa$ of $S$. 
Assume that $h_c(\Lbd(G))=1$ and $\Gr$ is homogeneous.
Then
$$
\Gr: \mbox{ non-amenable } \; \Longleftrightarrow \;\; \delta(G) < 1 .
$$
\end{corollary}

When $G$ is a non-trivial normal subgroup of a cocompact Fuchsian group,
Theorem~\ref{normalsubgroup} provides a similar consequence.
However, in Corollary~\ref{anyway-true}, we have only to assume that the graph 
$\Gr$ is homogeneous, which is weaker than the condition that $S$ itself admits 
the covering group action.

\begin{remark*}
Amenability of graphs is preserved by quasi-isometries 
(see e.g. \cite[Theorem~4.7]{woe00}).
If $\Gr$ is `quasi-homogeneous' in the sense that any vertex of $\Gr$ 
is mapped to a fixed vertex by a uniform quasi-isometry, 
then Corollary \ref{anyway-true} is still true under this weaker 
assumption on $\Gr$. 
\end{remark*}

Motivated by the statements of Theorem~\ref{brooks} and 
Corollary~\ref{anyway-true} we formulate the following 
conjecture.

\begin{conjecture}
\label{non-amenab-conj}
Let $S$ be a hyperbolic surface of infinite type having 
strongly bounded geometry, and assume $S = \B^2/G$, with $G$ a Fuchsian group. 
Let $\Gr$ be the graph associated to some uniform pants decomposition 
$\Pa$ of $S$. Then,
$$
\Gr: \mbox{ non-amenable } \; \Longleftrightarrow \;\; \delta(G) < \h(\Lbd(G)) .
$$
\end{conjecture}

\begin{remark*}
Our further strategy, pursued in Sections~\ref{patterson-for-X}
and~\ref{bd-type-cond}, will be to provide a condition under which 
$\dim_H(\Lbd(G))=\h(\Lbd(G))$ holds, in which case $\delta(G) < \dim_H(\Lbd(G))$ 
would follow from the above conjecture.
\end{remark*}

\begin{remark*}
Formulating the above statements in higher dimensions depends on defining a
higher dimensional analogue to the pants decomposition of a hyperbolic surface.
\end{remark*}

\section{Examples of dimension gap}
\label{examples-dim-gap}

In this section, we will illustrate two examples of a surface $S=\B^2/G$ 
for which the condition $\delta(G) < \h(\Lbd(G))$ is satisfied.
In these examples, the graph associated to some uniform pants decomposition of 
$S$ is non-amenable. Thus, these examples may serve as an evidence for the 
conjecture raised in the previous section.

The building block for our Riemann surface will be a pair of pants $P_\ell$ 
all of whose boundary geodesics $a_1$, $a_2$ and $a_3$ have the same length 
$\ell>0$. Let $H_\ell$ be the symmetric half of $P_\ell$, which is the 
right-angled hexagon having the sides $a_1^+$, $a_2^+$ and $a_3^+$ of length 
$\ell/2$. Let $o$ be the centre of $H_\ell$ and $r(\ell)>0$ be the distance 
between $o$ and the midpoint of $a_i^+$ ($i=1,2,3$). 
Then, a simple calculation using hyperbolic trigonometry yields that
$$
r(\ell)={\rm arccosh}\sqrt{\frac{4\sinh^2(\ell/4)+1}{3\sinh^2(\ell/4)}}.
$$
In particular, $r(\ell)$ is a decreasing function of $\ell$ and
$$
\lim_{\ell \to \infty}r(\ell)={\rm arccosh}\frac{2}{\sqrt{3}}=
\frac{\log 3}{2} \approx 0.55.
$$

\begin{example}\label{panty}
Take the pair of pants $P_\ell$ and glue a funnel (half annulus) along $a_1^+$ 
and two other pants $P_\ell$ along $a_2^+$ and $a_3^+$ without twist, that is, 
in such a way that their symmetric axes meet and their half hexagons 
$H_\ell$ are connected.
The resulting surface has four boundary components and we glue four other pants 
$P_\ell$ along them in the same manner. Continuing this process, we obtain a 
hyperbolic surface $S$ of infinite type
which has one boundary component at infinity.
Let $G$ be a Fuchsian group such that $S=\B^2/G$.

We choose the midpoints of the sides of all right-angled hexagons $H_\ell$ in 
$S$ and denote the set of all such midpoints by $\hat X \subset C(S)$.
Let $X$ be the inverse image of $\hat X$ under the projection $\B^2 \to S$.
Then, clearly, $X$ is a uniformly distributed set in the convex hull $H(\Lbd(G))$ 
of the limit set of $G$, and Proposition \ref{cc-entr-is-delta-X}
states that the convex core entropy $\h(\Lbd(G))$ is given by $\Delta(X)$.

We will estimate this $\Delta(X)$.
Let $x_0 \in X$ be a point such that its projection $\hat x_0 \in \hat X$ is the 
midpoint of the side on $\partial C(S)$ facing the boundary at infinity of $S$. 
The number of points in $\hat X$ within distance $2n\,r(\ell)$ from $\hat x_0$ 
can be estimated as
$\card (\hat X \cap B_{2nr(\ell)}(\hat x_0)) \geq 2^n$ for every $n \geq 0$. 
Hence we have
\begin{eqnarray*}
\Delta(X)&=&
\limsup_{R \to \infty} \frac{\log \card (X \cap B_R(x_0))}{R}\\
&\geq& \limsup_{n \to \infty} \frac{\log \card (\hat X 
\cap B_{2n\,r(\ell)}(\hat x_0))}{2n\,r(\ell)}\\
&\geq& \lim_{n \to \infty} \frac{\log (2^n)}{2n\,r(\ell)}
=\frac{\log 2}{2r(\ell)}.
\end{eqnarray*}
Since $\log 2 \approx 0.69$, we see that $\h(\Lbd(G))>1/2$ for all sufficiently 
large $\ell$.

On the other hand, we will estimate $\delta(G)$ by using the isoperimetric 
constant $h^{\Pa}(S)$ for the pants decomposition $\Pa$ trivially obtained from 
the construction of $S$. Let $W^\sharp$ be an arbitrary connected finite 
(not single) union of pants from $\Pa$.
Then the number of components of $\partial W^\sharp$ is always 
greater than 
the number of pants in $W^\sharp$. Hence
$$
h^{\Pa}(S)=\sup_{W^\sharp}\frac{A(W^\sharp)}{\ell(\partial W^\sharp)} \leq
\frac{2\pi}{\ell}.
$$
By Proposition \ref{matsuzaki} and Lemma \ref{h*-hP}, we have
$$
(1 \leq)\ h(S) \leq h^*(S)+1 \leq b h^{\Pa}(S)+1=\frac{2b \pi}{\ell}+1,
$$
where $b=b(\Pa)>0$ is a constant depending on $\Pa$.
As $\ell$ tends to infinity, $h(S) \to 1$, and thus $\lambda_0(S) \to 1/4$ by
Theorem~\ref{estimates}. Therefore, $\delta(G) \to 1/2$ by 
Theorem~\ref{els-pat-sul}. Hence, for all sufficiently large $\ell$, 
we have the strict inequality $\delta(G)<\h(\Lbd(G))$.
\end{example}

In the above example, the convex core entropy $\h(\Lbd(G))$ is actually 
equal to one in spite of the fact that the above arguments cannot produce this 
value. This is because $\dim_H(\Lbd(G)) \leq \h(\Lbd(G))$ by
Theorem~\ref{basicinequality} and $\dim_H(\Lbd(G))=1$ by \cite[Theorem 2]{mat00}.
(Actually the 1-dimensional Hausdorff measure of $\Lbd(G)$ is positive 
in this case.) On the other hand, since the graph associated to the uniform 
pants decomposition $\Pa$ is non-amenable, 
Theorem~\ref{non-amen-equiv-delta<1} yields $\delta(G)<1$. Thus,
$$
\delta(G)<1=\dim_H(\Lbd(G)) = \h(\Lbd(G))
$$
for this example, which also shows that the assumption of amenability is 
necessary in the statement of Theorem~\ref{amen-implies-equal}.

In what follows, we give another example of a situation where 
$\delta(G)<\h(\Lbd(G))$,
with $\h(\Lbd(G))$ being strictly less than one,
for which we also get control on the Hausdorff dimension of the limit set, in 
the sense that $\dim_H(\Lbd(G)) <1$.

\begin{example}
We first take the same surface $S$ as in Example~\ref{panty} but then we cut off
some of the pants in the following way. Let $\Gr$ be the graph associated
with the pants decomposition of $S$. This is a binary tree which is rooted at 
the vertex $v_0$ corresponding to the pair of pants that contains $\hat x_0$, 
and which branches into two edges at every vertex (the valency is three at 
every vertex except $v_0$). 
There are four vertices at depth 2 from $v_0$.
We remove one of the four vertices and the edge just terminating at this vertex
as well as the subtree rooted at this vertex. 
By this surgery, there remain 3 vertices $v_2$ at depth 2 from $v_0$. 
We repeat the same process beginning from the three $v_2$, namely,
remove one of the four vertices $v_4$ at depth 2 from each $v_2$. 
Thus there remain 9 vertices $v_4$ at depth 4 from $v_0$. 
We continue this surgery 
inductively to obtain a new graph $\Gr'$.
Then we rebuild a hyperbolic surface $S'$ from $\Gr'$ by replacing each 
vertex of $\Gr'$ with the pair of pants $P_\ell$ and by attaching funnels 
to the boundary components resulting from this cut-off process.
Let $G'$ be a Fuchsian group such that $S'=\B^2/G'$.

Let $\hat X' \subset C(S')$ be the set of 
all midpoints of the sides as before and
let $X'$ be the inverse image of $\hat X$ under the projection $\B^2 \to S'$.
Taking the base points $x_0 \in X'$ and $\hat x_0 \in \hat X'$
for the associated Poincar{\'e} series, we estimate $\Delta(X')$.
Let $d(\ell)$ be the distance between the midpoints of any two distinct 
sides of $H_\ell$, which thus satisfies $d(\ell) \leq 2r(\ell)$.
By the construction of $S'$, we have
$\card (\hat X \cap B_{2n\,d(\ell)}(\hat x_0)) \geq 3^n$ for every $n \geq 0$. 
Therefore,
\begin{eqnarray*}
\Delta(X')&=&
\limsup_{R \to \infty} \frac{\log \card (X' \cap B_R(x_0))}{R}\\
&\geq& \limsup_{n \to \infty} \frac{\log \card (\hat X' 
\cap B_{2n\,d(\ell)}(\hat x_0))}{2nd(\ell)}\\
&\geq& \lim_{n \to \infty} \frac{\log (3^{n})}{2n\,d(\ell)}
=\frac{\log 3}{2d(\ell)} \geq \frac{\log 3}{4r(\ell)}.
\end{eqnarray*}

However, if we use $r(\ell)$ as before, the fact that 
$\lim_{\ell \to \infty}r(\ell)=\log 3/2$ means that $\Delta(X')$
barely fails the desired estimate $\Delta(X') > 1/2$.
To overcome this problem, we have to deal with $d(\ell)$ directly.
A more complicated but straightforward calculation employing 
hyperbolic trigonometry yields
$$
d(\ell)=2\ {\rm arccosh}\sqrt{\frac{5\sinh^2(\ell/4)+1}{4\sinh^2(\ell/4)}}.
$$
In particular, $d(\ell)$ is a decreasing function of $\ell$ and
$$
\lim_{\ell \to \infty}d(\ell)=2\ {\rm arccosh}\frac{\sqrt{5}}{2} \approx 0.96.
$$
Hence we have $\h(S')=\h(\Lbd(G'))>1/2$ for all sufficiently large $\ell$.

The estimate of $\delta(G')$ is given by the same arguments as in 
Example~\ref{panty} and we have $\delta(G') \to 1/2$ as $\ell \to \infty$.
Thus, for all sufficiently large $\ell$, the strict inequality 
$\delta(G')<\h(\Lbd(G'))$ is obtained.
Note that, by construction, 
any point in the convex core $C(S')$ is within a bounded distance from the 
boundary $\partial C(S')$. In this case, Theorems~\ref{bdneighbouring} or \ref{porosityestimate}
imply $\h(\Lbd(G'))<1$.
\end{example}

\section{Patterson measures for uniformly distributed sets}
\label{patterson-for-X}

Fix a uniformly distributed set $X$ in the convex hull 
$H(\Lbd) \subset \B^{n+1}$ of a closed set $\Lbd \subset \S^n$. 
For any exponent $s$ greater than $\Delta=\Delta(X)$ and for any 
$z \in \B^{n+1}$, we consider the purely atomic measure $\mu_z^s$ on $\B^{n+1}$ 
given by
$$
\mu_z^s=\frac{1}{P^s(X,o)}\sum_{x \in X}e^{-sd(x,z)} {\bf 1}_{x}.
$$
Here, $o$ is the origin of $\B^{n+1}$ and ${\bf 1}_{x}$ is the point mass at $x$.
The total mass $\sigma(\mu_z^s)$ of $\mu_z^s$ is $P^s(X,z)/P^s(X,o)$.
Since $\sigma(\mu_o^s)=1$, we have
$$
e^{-sd(z,o)} \leq \sigma(\mu_z^s) \leq e^{sd(z,o)}.
$$

For the moment, we impose the following assumption on the divergence of
the Poincar\'e series at the exponent $\Delta$:
$$
P^\Delta(X,z)=\infty.
$$
We say that (the Poincar\'e series for) $X$ is of {\em divergence type} if this condition is satisfied.
Later, we will see that this is automatically satisfied for the 
Poincar\'e series for a uniformly distributed set $X$ of {\em bounded type} 
defined below (see Corollary~\ref{bounded-divergent}).

Similar to the usual Patterson construction of conformal measures on $\S^n$,
we take a decreasing sequence of exponents 
$\{s_j\}$ with $\lim_{j \to \infty} s_j=\Delta$.
Then, passing to a subsequence if necessary, we have a weak-$\ast$ limit of the 
sequence of purely atomic measures $\{\mu_z^{s_j}\}$ on the compact topological 
space $\overline{\B^{n+1}}=\B^{n+1} \cup \S^n$, which we denote by $\mu_z$. 
Note that this may be depending on the choice of the sequence of exponents 
$\{s_j\}$, but we expect that $\mu_z$ is uniquely determined 
independently of $\{s_j\}$.
It is well-known that when $X$ is the orbit of a convex cocompact Kleinian 
group, this uniqueness holds.
Although we cannot assume uniqueness of $\mu_z$ for uniformly distributed sets 
in general, it is easy to see that if $\{\mu_o^{s_j}\}$ has a weak-$\ast$ 
limit $\mu_o$ for $z=o$, then $\{\mu_z^{s_j}\}$ has a weak-$\ast$ limit 
$\mu_z$ for every $z \in \B^{n+1}$ by the same sequence of exponents $\{s_j\}$. 
Hence, whenever we consider a family of measures $\{\mu_z\}_{z \in \B^{n+1}}$, 
we always regard them as the weak-$\ast$ limits for the same sequence $\{s_j\}$.
We call $\{\mu_z\}_{z \in \B^{n+1}}$ a {\it Patterson measure} for 
the uniformly distributed set $X$ in the convex hull $H(\Lbd)$.

\begin{proposition}
Let $\{\mu_z\}_{z \in \B^{n+1}}$ be a Patterson measure for 
a uniformly distributed set $X$ in the convex hull $H(\Lbd)$ 
of a closed set $\Lbd \subset \S^n$.
Then the following properties are satisfied:
\begin{enumerate}
\item
The support of each $\mu_z$ is in $\Lbd$;
\item
The measures $\{\mu_z\}$ are absolutely continuous
to each other and the Radon-Nikodym derivative satisfies
$$
\frac{d\mu_z}{d\mu_o}(\xi)=k(z,\xi)^\Delta
$$
for any $z \in \B^{n+1}$ and for any $\xi \in \Lbd$. Here
$k(z,\xi)=(1-|z|^2)/|\xi-z|^2$ is the Poisson kernel and $\Delta=\Delta(X)$.
\end{enumerate}
\end{proposition}

This proposition says that the Patterson measure $\{\mu_z\}_{z \in \B^{n+1}}$ 
is a $\Delta$-conformal measure that is supported 
on $\Lbd$. A proof can be obtained in a similar manner to the case of Patterson 
measures for Kleinian groups.

The total mass of $\mu_z$ is given by
$$
\sigma(\mu_z)=\lim_{j \to \infty} \sigma(\mu^{s_j}_z)=
\lim_{j \to \infty} \frac{P^{s_j}(X,z)}{P^{s_j}(X,o)}.
$$
Then, for every $z \in \B^{n+1}$,
$$
e^{-\Delta d(z,o)} \leq \sigma(\mu_z) \leq e^{\Delta d(z,o)}.
$$

\begin{definition}
We say that a uniformly distributed set $X$ is of {\em bounded type}
if there exists a constant $\rho \geq 1$ such that
$$
\frac{\card (X \cap B_R(x))}{\card (X \cap B_R(o))} \leq \rho
$$
for every $x \in X$ and for every $R>0$. 
Here we assume that the origin $o \in \B^{n+1}$ belongs to $X$.
\end{definition}

When $X$ is the orbit of a convex cocompact Kleinian group, 
$X$ is clearly of bounded type. This is not the case for a uniformly distributed 
set $X$ in the convex hull of a geometrically
finite group with parabolic elements. Hence, 
it is useful to adapt the bounded type condition in such a way
that it also covers geometrically finite Kleinian groups.
This will be done later in this section.

It is easy to see that, if $X$ is of bounded type, then  
$P^{s}(X,x)/P^{s}(X,o)\leq \rho$
for every $x \in X$ and every 
$s>\Delta(X)$, and hence 
$$
\limsup_{s \searrow \Delta}\frac{P^{s}(X,x)}{P^{s}(X,o)} \leq \rho.
$$
This is also true for every $x$ in $H(\Lbd)$ by replacing $\rho$ with 
$K=\rho e^{\Delta M}$,
where $M>0$ is the constant in the definition of uniform distribution.
Then, by an elementary observation in hyperbolic geometry,
we see that the total mass $\sigma(\mu_z)$ of the Patterson measure $\mu_z$ is
uniformly bounded by $K$ for every $z \in \mathbb B^{n+1}$.

We will now deal with the situation that the Poincar\'e series of $X$ 
converges at the exponent $\Delta$.
We shall introduce a Patterson function and a modified Poincar\'e series
to construct a Patterson measure $\mu_z$, as in the case of the ordinary 
Poincar\'e series for Kleinian groups (see for instance \cite{ft06}).

\begin{definition}
For a uniformly distributed set $X$,
a continuous, non-decreasing function 
$h:(0,\infty) \to (0,\infty)$ is called a {\em Patterson function}
(or sometimes \emph{slowly varying function})
if the following conditions are satisfied:
\begin{enumerate}
\item
The {\em modified Poincar\'e series}
$$
P_h^s(X,z)=\sum_{x \in X} 
h(d(x,z))e^{-s d(x,z)}
$$
converges for $s>\Delta(X)$ and diverges 
for $s \leq \Delta(X)$;
\item
For every $\varepsilon >0$, there exists a constant $r_0>0$ such that
$h(t+r) \leq e^{\varepsilon t} h(r)$ for all $r \geq r_0$ and $t>0$;
\item
There exists a constant $C>0$ such that
$h(r+t) \leq C \, h(r) \, h(t)$ for all $r>0$ and $t>0$.
\end{enumerate}
\end{definition}

The same proof as in the case of orbits of Kleinian groups shows that a 
Patterson function exists for any uniformly distributed set $X$. Notice that,
if $P^\Delta(X,z)=\infty$, then the constant function
$h(t) \equiv 1$ can be a Patterson function for $X$.
Hence, by the same construction using the modified Poincar\'e series
instead of the ordinary one, we have a $\Delta$-conformal measure 
$\{\mu_z\}_{z \in \mathbb B^{n+1}}$ 
supported on $\Lambda$, which we again call a Patterson measure for $X$.

\begin{proposition}\label{modified-bounded}
Assume that the uniformly distributed set $X$ is of bounded type.
Then there exists a Patterson function $h$
for $X$ and a constant $\rho \geq 1$ 
such that the modified Poincar\'e series satisfies
$$
\frac{P_h^{s}(X,x)}{P_h^{s}(X,o)} \leq \rho
$$
for every $x \in X$ and for every $s>\Delta(X)$.
\end{proposition}

\begin{proof}
Since $X=\{x_i\}_{i \in \N}$ is of bounded type, there is by definition a constant 
$\rho \geq 1$ such that
$\card (X \cap B_R(x)) \leq \rho \, \card (X \cap B_R(o))$
for every $x \in X$ and for every $R>0$. This implies that
$$
\sum_{d(x_i,x) \leq R} h(d(x_i,x))e^{-s d(x_i,x)} \leq 
\rho \sum_{d(x_j,o) \leq R} h(d(x_j,o))e^{-s d(x_j,o)}
$$
for every $R>0$.
Taking the limit as $R \to \infty$, we have the desired estimate.
\end{proof}

By the same argument as before, we see from Proposition \ref{modified-bounded} 
that, if $X$ is of bounded type, then the total mass $\sigma(\mu_z)$ 
of the Patterson measure $\mu_z$ for $X$ is uniformly bounded by $K$
independently of $z \in \mathbb B^{n+1}$.

\begin{proposition}
The Patterson measure $\mu_z$ for $X$ has no atom in $\Lbd$ if $X$ is of 
bounded type.
\end{proposition}

\begin{proof}
If $\mu_z$ has an atom at $\xi \in \Lbd$, then 
$$
\sigma(\mu_{z}) \geq \mu_{z}(\{\xi\})=k(z,\xi)^\Delta\mu_o(\{\xi\}),
$$
and this diverges to $\infty$ when $z \in \mathbb B^{n+1}$ tends to $\xi$
radially. However, this contradicts the fact that the total mass 
$\sigma(\mu_{z})$ is uniformly bounded.
\end{proof}

One of the very useful properties of Patterson measures for groups is
given by the so-called \emph{Sullivan shadow lemma}. 
In our case we obtain a partial result, 
which will be sufficient for further considerations.
 
\begin{definition}
For a point $x \in \mathbb B^{n+1}$ different from the origin $o$,
the antipodal point $x^\ast \in \mathbb S^n$ is the end point of the 
geodesic ray from $x$ passing through $o$. Let $B(x,\ell)$ be
a hyperbolic ball with the centre $x$ and radius $\ell>0$.
The \emph{shadow} $S(x,\ell)$ of the ball $B(x,\ell)$ is the set of points 
$y \in \mathbb S^n$ so that the geodesic from $x^\ast$ to $y$
intersects $B(x,\ell)$.
\end{definition}

The diameter of a shadow $S(x,\ell)$ is measured using the spherical distance on 
$\mathbb S^n$ and is denoted by $\diam S(x,\ell)$. In the case of the orbits
of Kleinian groups, the shadow lemma asserts that the Patterson measure 
of a shadow and the $\Delta$-power of its diameter are comparable. 
In our case, we have a similar statement as follows.

\begin{theorem}\label{shadow}
Let $X$ be a uniformly distributed set of bounded type and
$\mu_z$ the Patterson measure for $X$ with uniformly bounded total mass 
$\sigma(\mu_z) \leq K$.
Fix a radius $\ell>0$.
Then there exists a constant $A>0$ depending only on $K$ and $\ell$ such that
the diameter of a shadow $S(x,\ell)$ for any $x \in X\setminus \{o\}$ satisfies
$$
A\,(\diam S(x,\ell))^\Delta \geq \mu_o(S(x,\ell)).
$$
\end{theorem}

\begin{proof}
Let $\gamma_x$ be a conformal automorphism of $\mathbb B^{n+1}$ 
that sends $x$ to $o$. Its linear stretching factor 
$|\gamma'_x(\xi)|$ at $\xi \in \mathbb S^n$  
coincides with the Poisson kernel $k(x,\xi)$.
We will estimate $|\gamma'_x(\xi)|$ on $S(x,\ell)$.
For any $\xi \in S(x,\ell)$, set $\eta=\gamma_x(\xi)$. Then 
$$
|\gamma'_x(\xi)|=\frac{1}{|(\gamma^{-1}_x)'(\eta)|}=
\frac{|\eta-\gamma_x(o)|^2}{1-|\gamma_x(o)|^2}
=\frac{|\eta-\gamma_x(o)|^2}{1-|x|^2}.
$$
Here we have constants $a>0$ and $b>0$ depending only on $\ell$ such that
$$
1-|x|^2 \leq a\ \diam S(x,\ell) \;\;\mbox{ and} 
\quad|\eta-\gamma_x(o)|^2 \geq b
$$
for any $\eta$ in $\gamma_x(S(x,\ell))$. Therefore,
$$
k(x,\xi)=|\gamma'_x(\xi)| \geq \frac{b}{a\ \diam S(x,\ell)},
$$
and hence
\begin{eqnarray*}
K \geq \sigma(\mu_x) &\geq& \int_{S(x,\ell)}d\mu_x(\xi)\\
&=& \int_{S(x,\ell)}k(x,\xi)^\Delta d\mu_o(\xi)\\
& \geq & \left(\frac{b}{a}\right)^\Delta 
\frac{\mu_o(S(x,\ell))}{(\diam S(x,\ell))^\Delta}.
\end{eqnarray*}
By taking $A=K(a/b)^\Delta$, we obtain the assertion.
\end{proof}

Employing the shadow lemma (Theorem \ref{shadow}) and condition (i) 
of the definition of uniformly distributed sets, we can obtain the 
lower estimate of the Hausdorff dimension of 
$\Lambda$ by the critical exponent $\Delta=\Delta(X)$.
This is a generalisation of the corresponding result for the Hausdorff dimension 
of limit sets of convex cocompact Kleinian groups.

\begin{theorem}\label{hausdorff}
Let $\Lbd$ be a closed subset on $\mathbb S^n$ and assume that
there is a uniformly distributed set $X$ of bounded type
in the convex hull $H(\Lbd)$.
Then the $\Delta$-dimensional Hausdorff measure of $\Lbd$ is positive.
In particular, $\dim_H (\Lbd) \geq  \Delta(X)$, and hence we have that 
$\dim_H (\Lbd) =\Delta(X) = \h(\Lbd)=\dim_B(\Lbd)$.
\end{theorem}

\begin{proof}
By usual arguments (see Nicholls \cite[Section 4.6]{nic89})
relying on Theorem \ref{shadow}, the $\Delta$-dimensional Hausdorff measure $m$
satisfies $m(\Lbd) \geq c \, \mu_o(\Lbd)$ for some constant $c>0$. 
Since $\mu_o(\Lbd)>0$, we have that $m(\Lbd)>0$. 
This implies that $\dim_H (\Lbd) \geq  \Delta(X)$. 
By Theorem~\ref{basicinequality} and Proposition~\ref{cc-entr-is-delta-X}, 
we conclude $\dim_H (\Lbd)=\Delta(X)=\h(\Lbd)=\overline{\dim}_B(\Lbd)$.
Thus, $\dim_B(\Lbd)$ exists and equals all of these invariants.
\end{proof}

\begin{remark*}
We have relied on the shadow lemma (Theorem~\ref{shadow}) in order to show 
Theorem~\ref{hausdorff}.
An alternative proof may be possible by using the arguments in Bishop and 
Jones~\cite{bj97} which give $\dim_H(\Lbd_c(G))=\delta(G)$ for any 
non-elementary Kleinian group $G$. However, these arguments require some
degree of homogeneity of $X$ and it is not easy to deduce this only from 
the bounded type condition for $X$. 
In other words, it is only the number of points in $X$ within some distance 
from some point $x \in X$ which is controlled by the bounded type condition, 
whereas all `geometric characteristics' (i.e. both distance and direction) 
of $X$ viewed from $x \in X$ must be almost the same in order to apply 
Bishop and Jones' arguments.
\end{remark*}

From Theorem~\ref{hausdorff}, we see that the bounded type property implies 
the divergence type property of the Poincar\'e series for a uniformly 
distributed set $X$.

\begin{corollary}
\label{bounded-divergent}
Any uniformly distributed set $X$ of bounded type
in $H(\Lbd)$ satisfies $P^\Delta(X,o)=\infty$ for $\Delta=\Delta(X)$.
\end{corollary}

\begin{proof}
For $X=\{x_i\}_{i \in \N}$, take the shadows $S(x_i,\ell) \subset \S^n$, 
$i \in \N$, of Euclidean diameter comparable to $e^{-d(x_i,o)}$ with constant of 
comparability only depending on $\ell$. Since
$$
\Lbd = \limsup \{S(x_i,\ell): i \in \N \},
$$
the convergence of $P^\Delta(X,o)$ implies, by means of 
the Borel-Cantelli lemma, that 
the $\Delta$-dimensional Hausdorff measure $m(\Lbd)$ of $\Lbd$ is zero. 
However, Theorem \ref{hausdorff} gives $m(\Lbd)>0$.
This shows that $P^\Delta(X,o)$ diverges.
\end{proof}

Conversely, one may ask whether the divergence type property of the 
Poincar\'e series for $X$ implies 
$\dim_H (\Lbd) = \h(\Lbd)=\dim_B(\Lbd)$.
There are closed sets $\Lbd \subset \S^n$ such that
$\dim_H (\Lbd) < \overline \dim_B (\Lbd)$,
for example those given by McMullen \cite{mcm84}.
This tells us that, by Theorem \ref{hausdorff}, not
all closed sets $\Lbd \subset \S^n$ necessarily
contain a uniformly distributed set $X$
of bounded type in their convex hull $H(\Lbd)$.

In the remainder of this section, we shall attempt to extend the above arguments
to the case where $H(\Lambda)$ contains horoballs. When $\Lambda=\Lambda(G)$ is 
the limit set of some Kleinian group $G$, this occurs mainly in the case where $G$ has 
parabolic elements and $H(\Lambda(G)) (\neq \B^{n+1})$ contains no uniformly 
distributed set of bounded type.
However, we want to include, say, the limit set of a geometrically finite 
Kleinian group with parabolics in the category of `preferred' closed sets.
Hence we weaken the definition of the bounded type condition in the following way.

\begin{definition}
\label{weaklybounded}
We say that a uniformly distributed set $X$ is of {\em weakly bounded type}
if there exist a constant $\rho \geq 1$ and a family of 
mutually disjoint horoballs $\{D_p\}_{p \in \Phi}$ in $\B^{n+1}$ 
with set of tangency points $\Phi \subset \S^n$ such that
$$
\frac{\card (X \cap B_R(x))}{\card (X \cap B_R(o))} \leq \rho
$$
for every $x \in X'$ and for every $R>0$. 
Here $X'=X \cap (\B^{n+1} \setminus \bigcup_{p \in \Phi} D_p)$ and
we assume that the origin $o \in \B^{n+1}$ belongs to $X'$.
\end{definition}

Under this assumption, we can also conclude that the Poincar\'e series for 
the uniformly distributed set $X$ is of divergence type if the corresponding 
Patterson measure has no atom on the set $\Phi$ of tangency points of horoballs. 
In this case, we are led to consider the packing dimension $\dim_P(\Lbd)$ of 
$\Lbd$ instead of the Hausdorff dimension.

\begin{theorem}\label{wb-packing}
Suppose that a uniformly distributed set $X$ in the convex hull of
$\Lbd \subset \S^n$ is of weakly bounded type and the Patterson measure 
$\mu_z$ for $X$ has no atom on the tangency points $p \in \Phi \subset \S^n$
of the horoballs $\{D_p\}_{p \in \Phi}$ appearing in the definition
of weak boundedness. 
Then the packing dimension $\dim_P(\Lbd)$ of $\Lbd$ coincides with 
$\Delta(X)=h_c(\Lbd)=\overline \dim_B(\Lambda)$. 
Moreover, $P^\Delta(X,o)=\infty$ for $\Delta=\Delta(X)$.
\end{theorem}

\begin{proof}
Since the Patterson measure $\mu_z$ has no atom on any $p \in \Phi$, 
we have $\mu_o(\Lbd')>0$ for $\Lbd'=\Lbd \setminus \Phi$.
Using Theorem \ref{shadow}
for $X'=X \cap (\B^{n+1} \setminus \bigcup_{p \in \Phi} D_p)$
and applying the same argument as in Sullivan \cite{sul84}, 
we see that the $\Delta$-dimensional packing measure $\nu$
satisfies $\nu(\Lbd') \geq c \mu_o(\Lbd')$ for some constant $c>0$. 
Hence $\nu(\Lbd)>0$, which implies that $\dim_P(\Lbd) \geq \Delta(X)=h_c(\Lbd)$.
On the other hand, since the packing dimension $\dim_P(\Lambda)$ is always less 
than or equal to the upper box-counting dimension $\overline \dim_B(\Lambda)$, 
we have $\dim_P (\Lbd)=h_c(\Lbd)$.

The proof of the latter statement is carried out similarly to 
Corollary \ref{bounded-divergent}
by using $X'=\{x_i\}_{i \in \N}$ instead of $X$.
Namely, take the shadow $S(x_i,\ell) \subset \S^n$ of Euclidean diameter comparable to $e^{-d(x_i,o)}$
for each $i$, and represent
$\Lbd' = \limsup \{S(x_i,\ell): i \in \N \}$.
By the Borel-Cantelli lemma,
the convergence $P^\Delta(X',o) \leq P^\Delta(X,o)<\infty$ implies
that $\nu(\Lbd')=0$. Hence $\nu(\Lbd')>0$ yields $P^\Delta(X,o)=\infty$.
\end{proof}

A sufficient condition for the Patterson measure to have no atom on $\Phi$
can be easily given in the usual way when $\Lbd$ is the limit set of some 
Kleinian group.

\begin{corollary}\label{wbimplydivergence}
Let $G$ be a non-elementary Kleinian group acting on $\B^{n+1}$ all of whose 
parabolic fixed points are bounded. If a 
uniformly distributed set $X$ in $H(\Lbd(G))$ is of 
weakly bounded type for the set $\Phi$ of all parabolic fixed points of $G$, 
then $\dim_P(\Lbd(G))=\Delta(X)$ and $P^\Delta(X,o)=\infty$. 
\end{corollary}

\begin{proof}
We can choose $X'=X \cap (\B^{n+1} \setminus \bigcup_{p \in \Phi} D_p)$ 
invariant under $G$. Then, a similar argument as in the case of Patterson 
measures defined by an orbit of $G$ gives that the Patterson measure 
$\mu'_z$ for $X'$ has no atom on the bounded parabolic fixed points of $G$.
The statement therefore follows from the proof of the previous theorem.
\end{proof}

\begin{remark*}
It is known that, if $\overline \dim_B(\Lbd \cap U)=\overline \dim_B(\Lbd)$
for every open set $U \subset \S^n$ that intersects $\Lbd$,
then $\dim_P(\Lbd)=\overline \dim_B(\Lbd)$ 
(see e.g. Falconer \cite[Corollary 3.9]{fal90}).
This also yields the first assertion of the above corollary.
\end{remark*}

\section{Kleinian groups of bounded type}
\label{bd-type-cond}

For a Kleinian group $G$ acting on $\mathbb B^{n+1}$,
we consider the problem whether there is
a uniformly distributed set $X$ of bounded type
in the convex hull $H(\Lbd(G))$ of the limit set $\Lbd(G) \subset \S^n$.
When this is the case, we say that the Kleinian group $G$ is of bounded type.
If $G$ is of bounded type, 
then $\dim_H(\Lbd(G))=\h(\Lbd(G))=\dim_B(\Lbd(G))$ by Theorem \ref{hausdorff}.
Any convex cocompact Kleinian group $G$ is of bounded type
because the orbit $X$ of any point in $H(\Lbd(G))$ is
a $G$-invariant uniformly distributed set, and in particular of bounded type.

First, we show that the bounded type property is inherited to normal subgroups.

\begin{proposition}
If $G$ is a non-trivial normal subgroup of a convex cocompact Kleinian group
$\hat G$, then $G$ is of bounded type. More generally, any 
non-trivial normal subgroup $G$ of a Kleinian group
$\hat G$ of bounded type is also of bounded type.
\end{proposition}

\begin{proof}
Let $X$ be a uniformly distributed set of bounded type in
the convex hull $H(\Lbd(\hat G))$.
Since $\Lbd(\hat G)=\Lbd(G)$, their convex hulls are the same,
and hence $X$ is also a uniformly distributed set for $\Lbd(G)$.
\end{proof}

Next, we investigate if certain deformations of Kleinian groups preserve
the bounded type property. Recall that a homeomorphism 
$f:\B^{n+1} \to \B^{n+1}$ is a $K$-quasi-isometry for $K \geq 1$ if
there is a constant $L \geq 0$ such that
$$
\frac{1}{K}d(x,y)-\frac{L}{K} \leq d(f(x),f(y)) \leq K d(x,y)+L
$$
for any $x$ and $y$ in $\B^{n+1}$. Also, a $1$-quasi-isometry is 
referred to as a {\it rough isometry}.

\begin{proposition}\label{roughiso}
Let $G'$ be a rough-isometric deformation of a Kleinian group $G$
of bounded type. Then $G'$ is also of bounded type.
\end{proposition}

\begin{proof}
Assume that $G'$ is obtained from $G$ by a rough-isometric automorphism  
$f:\B^{n+1} \to \B^{n+1}$ such that $G'=fGf^{-1}$.
Given a uniformly distributed set $X$ of bounded type for $\Lbd(G)$
satisfying by definition
$$
\frac{\card (X \cap B_R(x))}{\card (X \cap B_R(o))} \leq \rho
$$
for every $x \in X$ and for every $R>0$,
we consider $X'=f(X)$, which may not be in the convex hull
$H(\Lbd(G'))$ of $G'$. By composing, if necessary, with another 
rough-isometric automorphism conjugating $G'$ onto itself,
we may assume that $X'$ is a uniformly distributed set in $H(\Lbd(G'))$.
Since $f$ satisfies 
$$
d(x,y)-L \leq d(f(x),f(y)) \leq d(x,y)+L
$$
for some $L \geq 0$, we have
\begin{eqnarray*}
\card (f(X) \cap B_R(f(x))) &\leq& \card (X \cap B_{R+L}(x));\\
\card (f(X) \cap B_R(f(o))) &\geq& \card (X \cap B_{R-L}(o)).
\end{eqnarray*}
This implies that
\begin{eqnarray*}
\frac{\card (f(X) \cap B_R(f(x)))}{\card (f(X) \cap B_R(f(o)))} &\leq&
\frac{\card (X \cap B_{R+L}(x))}{\card (X \cap B_{R-L}(o))}\\
&=&
\frac{\card (X \cap B_{R+L}(x))}{\card (X \cap B_{R+L}(o))}\cdot
\frac{\card (X \cap B_{R+L}(o))}{\card (X \cap B_{R-L}(o))}\\
&\leq&
\rho \, c \, e^{2Ln}
\end{eqnarray*}
for some constant $c>0$,
which shows that $X'=f(X)$ is of bounded type.
\end{proof}

This observation leads us to the following problem. 

\begin{problem}
Let $G'$ be a quasi-isometric deformation of a Kleinian group $G$
of bounded type. Is then $G'$ also of bounded type?
\end{problem}

So far, we have seen that the following classes of Kleinian groups $G$ are of 
bounded type and hence satisfy $\dim_H(\Lbd(G))=\h(\Lbd(G))=\dim_B(\Lbd(G))$:
\begin{enumerate}
\item
Kleinian groups of the first kind ($\Lbd(G)=\S^n$);
\item
Convex cocompact Kleinian groups;
\item
non-trivial normal subgroups of Kleinian groups of bounded type;
\item
rough-isometric deformations of Kleinian groups of bounded type.
\end{enumerate}
In what follows, we exhibit an example which does not belong to the 
above classes but still is of bounded type.

The construction of this example is based on the following theorem
describing properties of
the Diestel-Leader graph $DL(m,n)$ for $m,n \geq 2$ and $m \neq n$,
which is proved by Eskin, Fisher and Whyte \cite{efw}.

\begin{theorem}\label{Diestel-Leader}
There exists a locally finite graph $\Gr=(V,E)$ such that the group of 
isometries of $\Gr$ acts transitively on the vertex set $V$ but $\Gr$ is not
quasi-isometric to the Cayley graph of any finitely generated group.
\end{theorem}

We also need the following result, which is given by a similar 
argument to Proposition \ref{roughiso}.

\begin{proposition}\label{uniformL}
Let $X$ be a uniformly distributed set in the convex hull 
$H(\Lbd)$ of a closed set $\Lbd\subset \S^n$. 
Assume that there exists a uniform constant $L \geq 0$ such that, 
for any two points $x_1$ and $x_2$ in $X$, there is a rough-isometric 
automorphism $f$ of $\B^{n+1}$ with additive constant $L$ 
satisfying $f(X)=X$ and $f(x_1)=x_2$. Then $X$ is of bounded type.
\end{proposition}

\begin{proof}
For any two points $x_1$ and $x_2$ in $X$, we have
$$
\card (X \cap B_R(x_2)) \leq \card (X \cap B_{R+L}(x_1)),
$$
and hence
$$
\frac{\card (X \cap B_R(x_2))}{\card (X \cap B_R(x_1))} \leq
\frac{\card (X \cap B_{R+L}(x_1))}{\card (X \cap B_R(x_1))}
\leq \rho \, c \, e^{Ln},
$$
where $\rho \geq 1$ and $c>0$ are the same constants as in the proof of 
Proposition~\ref{roughiso}.
This implies that $X$ is of bounded type.
\end{proof}

\begin{example}
We take a graph $\Gr$ as in Theorem \ref{Diestel-Leader} and let $k$ be
the valency of each vertex of $\Gr$. Consider a hyperbolic surface $Q$
with $k+1$ geodesic boundary components that is homeomorphic to a $(k+1)$-holed 
sphere. One boundary component of $Q$ is reserved for gluing a funnel 
(half annulus) and the other $k$ components are used for the following 
construction tracing $\Gr$. Put a copy of $Q$ on each vertex of $\Gr$ and 
fit them together according to the structure of $\Gr$ by gluing adjacent 
geodesic boundaries of $Q$. The resulting hyperbolic surface is denoted by $S$. 
It is easy to see that $\Gr$ and the convex core $C(S)$ of $S$ are 
quasi-isometric to each other.

We first show that $S$ is not quasi-isometric to any normal cover of any
convex compact surface, that is, a hyperbolic surface given by a 
convex cocompact Fuchsian group. Suppose to the contrary that $S$  
is quasi-isometric to a normal cover $\Sigma$ of a convex compact 
surface $\hat \Sigma$. It follows that the convex cores $C(S)$ and 
$C(\Sigma)$ are quasi-isometric. Then $C(\Sigma)$ is quasi-isometric to the 
Cayley graph of the covering group isomorphic to 
$\pi_1(\hat \Sigma)/\pi_1(\Sigma)$.
Since $\pi_1(\hat \Sigma)$ is finitely generated, so is 
$\pi_1(\hat \Sigma)/\pi_1(\Sigma)$.
This implies that $\Gr$ is quasi-isometric to the Cayley graph of 
a finitely generated group, which contradicts Theorem \ref{Diestel-Leader}.

Represent $S=\B^2/G$ by a Fuchsian group $G$. Choose a point $\hat x$ 
in each $Q$ and let $\hat X$ be the set of all such points $\hat x$ in $C(S)$. 
Then, define $X$ as the inverse image of $\hat X$ under the universal cover 
$\B^2 \to S$, which is thus a uniformly distributed set in the convex hull 
$H(\Lbd(G))$. We will prove that there is a uniform constant $L \geq 0$ such that,
for any two points $x_1$ and $x_2$ in $X$, there is a rough-isometry 
$f: \B^2 \to \B^2$ with additive constant $L$ satisfying
$f(x_1)=x_2$, $f(X)=X$ and $fGf^{-1}=G$.
Then, by Proposition \ref{uniformL},
we see that $X$ is of bounded type, and hence $G$ is of bounded type.
If $x_1$ and $x_2$ project to the same point in $\hat X$ by the universal cover,
then $f$ can be taken as an element of the covering group, which is an isometry 
of $\B^2$.
Hence we may assume that $x_1$ and $x_2$ project to distinct points $\hat x_1$ 
and $\hat x_2$ in $\hat X$.

Let $v_1$ and $v_2$ be the vertices of $\Gr$ corresponding to $\hat x_1$ 
and $\hat x_2$ respectively.
Since there is an isometry $\sigma$ of $\Gr$ sending $v_1$ to $v_2$, 
we can take a piecewise isometry $\hat f_*$ of $S$ preserving the structure of 
the union of $Q$ and mapping $\hat x_1$ to $\hat x_2$.
Then we modify $\hat f_*$ to an automorphism $\hat f:S \to S$
so that the action on $\Gr$ induced by $\hat f$ is the same as $\sigma$.
Actually, on each $Q \subset S$ associated with a vertex $v \in \Gr$ with the 
adjacent edges $\{e_1, \ldots, e_k\}$, the difference between
$\sigma$ and $\hat f_*$ appears as a permutation of the set of edges 
$\{e_1, \ldots, e_k\}$.
Hence we have only to modify $\hat f_*$ to $\hat f$ on each $Q$ according to 
the permutation.
Since the number of such permutations is finite, there is a uniform constant 
$K \geq 1$ such that the modified map $\hat f$ can be chosen as a
$K$-quasi-isometry on each $Q$.
We can find similar arguments in Bonfert-Taylor et al. \cite{bcst}.
Since the action on $\Gr$ induced by $\hat f$ is isometric, we see that such
$\hat f$ is a rough-isometry with an additive constant $L$ depending only on 
$S$ and $K$.
Lifting $\hat f$ to $\B^2$, we obtain the required rough-isometry $f$.
\end{example}

Although the above construction might give the impression that only some 
special Kleinian groups beside the ones in the previous list are of bounded type,
we expect that there are a large number of such groups.
We formulate our expectation as follows. Note that, if we want to allow a Kleinian group
to have parabolic elements, we have to replace the bounded type condition 
with the weakly bounded type condition given in Definition \ref{weaklybounded}.
(A Kleinian group is of weakly bounded type if it has this property.)

\begin{conjecture}
\label{bd-from-bdry-bd-type}
If the convex hull $H(\Lbd(G))$ of the limit set of a Kleinian group~$G$ 
of divergence type has the (weak) tightness property, 
then $G$ is of (weakly) bounded type.
\end{conjecture}

\end{document}